\documentclass[a4paper,oneside,10pt]{article}

\usepackage{amsmath}%
\usepackage{amssymb}%
\usepackage[latin1]{inputenc}
\usepackage[numbers]{natbib}
\usepackage{enumitem}
\usepackage{euscript}

\usepackage{geometry}
\newtheorem{theorem}{Theorem}[section]

\newtheorem{corollary}[theorem]{Corollary}

\newtheorem{definition}[theorem]{Definition}

\newtheorem{lemma}[theorem]{Lemma}

\newtheorem{proposition}[theorem]{Proposition}
\newtheorem{remark}[theorem]{Remark}

\newenvironment{proof}[1][Proof]{\textbf{#1.} }{\ \rule{0.5em}{0.5em}}

\newcommand{\refeqn}[1]{(\ref{#1})}
\newcommand{\virgolette}[1]{``#1''}
\newcommand{\cinf}[0]{C^{\infty}}

\newcommand{\matr}[0]{\operatorname{Mat}}
\newcommand{\spann}[0]{\operatorname{span}}
\newcommand{\Iso}[0]{\operatorname{Iso}}
\newcommand{\acc}[1]{\`{#1}}
\newcommand{\argg}[0]{\operatorname{arg}}

\begin{document}

\title{{\bf Weak symmetries of stochastic differential equations driven by semimartingales with jumps}}
\author{Sergio Albeverio\thanks{Institut f\"ur Angewandte Mathematik, Rheinische Friedrich-Wilhelms-Universit\"at Bonn, Endenicher Allee 60, Bonn, Germany, \emph{email: albeverio@uni-bonn.de}}, Francesco C. De Vecchi\thanks{Institut f\"ur Angewandte Mathematik, Rheinische Friedrich-Wilhelms-Universit\"at Bonn, Endenicher Allee 60, Bonn, Germany, \emph{email: francesco.devecchi@uni-bonn.de}}, Paola Morando\thanks{DISAA, Universit\`a degli Studi di Milano, via Celoria 2, Milano, Italy, \emph{email: paola.morando@unimi.it}} and Stefania Ugolini\thanks{Dipartimento di Matematica, Universit\`a degli Studi di Milano, via Saldini 50, Milano, Italy, \emph{email: stefania.ugolini@unimi.it}}}
\date{}
 \maketitle

\begin{abstract}
Stochastic symmetries and related invariance properties of finite dimensional SDEs driven by general  c\acc{a}dl\acc{a}g  semimartingales taking values in Lie groups are defined
and investigated.  The considered set of SDEs, first introduced by S. Cohen,
includes affine and Marcus type SDEs as well as smooth SDEs driven by  L\'evy processes and iterated random maps. A natural extension to this general setting of reduction and reconstruction theory for symmetric SDEs is provided. Our theorems imply as special cases non trivial invariance results concerning a class of affine iterated random maps as well as symmetries for numerical schemes (of Euler and Milstein type) for Brownian motion driven SDEs.
\end{abstract}



\section{Introduction}

The study of symmetries and invariance properties of ordinary and partial differential equations (ODEs and PDEs, respectively) is a classical and
well-developed
research field (see \cite{Bluman1989,Gaeta1994,Olver1993,Stephani1989})
and provides a powerful tool for both computing some explicit solutions to the equations and analyzing their qualitative behavior.\\
The study of invariance properties of finite or infinite dimensional stochastic differential equations (SDEs)
is, in comparison, less developed and a systematic study could be fruitful from both the practical and the theoretical point of view.\\
The knowledge of some closed formulas is important in many applications of stochastic processes since it  permits to
develop faster and cheaper numerical algorithms for the simulation of the process or to evaluate interesting quantities related to it. Moreover,
the use of closed formulas
 allows the application of simpler statistical methods for the calibration of models (this is the reason for the popularity of affine
models in mathematical finance, see e.g. \cite{Cuchiero2011,Filipovic2003}, or of the Kalman filter and its generalizations in the theory of
stochastic filtering, see e.g. \cite{Bain2009}). The presence of symmetries and invariance properties is
 a strong clue for the possibility of closed formulas (see, for example, \cite{Craddock2009,Craddock2007,Craddock2012},
 where classical infinitesimal symmetry techniques are used for finding fundamental solutions of some diffusion processes applied in mathematical
 finance, or \cite{DeLara1997(1),DeLara1997(2),DM2016}, where geometrical methods based on Lie algebras are used to find new finite
 dimensional stochastic filters).\\
The investigation of invariance properties is relevant also from a theoretical point of view, in particular when stochastic processes are discussed in a
geometrical framework. Some interesting examples of this approach are the study of  L\'evy  processes on Lie groups \cite{Albeverio2007,Liao2004},
the geometric description of
stochastic filtering (see \cite{Elworthy2010}, where invariant diffusions on fiber bundles are discussed), and the study of variational
stochastic systems \cite{Cruzeiro2016,Holm2015,Zambrini2015}.\\

In this paper we apply Sophus Lie original ideas to the study of stochastic symmetries of a finite dimensional SDE driven by
general c\acc{a}dl\acc{a}g semimartingales taking values in  Lie groups. In particular, we introduce a group of transformations which change both
the processes solving the considered SDE and its driving noise, and  transform correspondingly the coefficients of the SDE. Therefore,  we look for the subgroup of all transformations which leave
invariant the set of solutions of a given SDE.\\
 In order to clarify the novelty of our study we describe, without claiming to be exhaustive, some
previous results on the same problem.
There are essentially two natural approaches to the description of the symmetries of a SDE. The first one, applied when the solution processes are
Markovian semimartingales, consists in studying the invariance properties of the generator of the SDE solutions (which is an analytical object).
This approach, used by Glover et al. \cite{Glover1991,Glover1998,Glover1990}, Cohen de Lara \cite{DeLara1991,DeLara1995} and Liao
\cite{Liao1992,Liao2009}, deals with a large group of transformations involving both a general spatial transformation and a
solution-dependent stochastic time change.\\
The second approach, mainly applied to Brownian-motion-driven SDEs, consists in restricting the attention to a suitable set of transformations and directly applying
a natural notion of symmetry, closely inspired by the ODEs case (see Gaeta et al.
\cite{Gaeta2000,Gaeta1999}, Unal \cite{Unal2004}, Srihirun, Meleshko and Schulz
\cite{Srihirun2006}, Fredericks and Mahomed \cite{Fredericks2007}, Kozlov \cite{Kozlov2010(1),Kozlov2010(2)} for SDEs driven by Brownian motion (see also \cite{Gaeta2017} for a review on this subject) and L\'azaro-Cam\'i and Ortega \cite{Cami2009} for SDEs driven by general continuous semimartingales). \\
Both approaches have their strengths and weakness as for example, the first method  permits to treat a very general family of processes (all
Markovian processes on a metric space) and a large class of transformations with interesting applications  (see
\cite{Glover1990,Liao2009}), but the explicit calculation of the symmetries is quite difficult in the non-diffusive case.
Conversely, the second approach
 allows us to face the non-Markovian setting (see \cite{Cami2009}) and  permits easy explicit calculations. In particular, in this framework, it
 is possible   to
get the \emph{determining equations}, that are a set of first order PDEs which uniquely characterizes the  symmetries of the SDE. As the first approach,
also the second one has interesting
applications (see, e.g., \cite{Cami2009,DMU2}), even though,  until now, it has been confined to the case of continuous semimartingales
and usually considers a  family of transformations which is smaller than the family considered by the first method.\\

In this paper, aiming at  reducing the gap between the two approaches we just mentioned, we propose a possible foundation of the concept of symmetry for general SDEs
and we extend the methods introduced in \cite{DMU1} where, despite working in the setting of the second approach, we introduced a large family of transformations which allowed us to obtain all the symmetries of the first setting for Brownian-motion-driven SDEs.\\
Let us recall that in \cite{DMU1} we  considered a SDE as a pair $(\mu(x),\sigma(x))$, where $\mu$ is the drift and $\sigma$ is the diffusion coefficient
defined on a manifold $M$ and we called solution to the SDE $(\mu,\sigma)$ a pair $(X,W)$, where $X$ is a semimartingale on $M$ and $W$ is an
 $n$-dimensional
 Brownian
motion. A stochastic transformation, according to \cite{DMU1}, is a triple  $T=(\Phi,B,\eta)$, where $\Phi$ is a diffeomorphism of $M$, $B$ is a $X_t$-dependent rotation and
$\eta$ is a $X_t$-dependent density of a stochastic time change. The transformation $T$ induces an action $E_T$ on the SDE $(\mu,\sigma)$ and an action
$P_T$ on the process $(X,W)$. The operator $P_T$ acts on the process $(X,W)$ changing the semimartingale
$X$ by the diffeomorphism $\Phi$ and the time change $\int_0^t{\eta dt}$, and on the Brownian motion $W$ by the rotation $B$ and the
same time change. Since the Brownian motion is invariant with respect to both rotations and time rescaling, the process $P_T(X,W)$ is composed by
a semimartingale on $M$ and a new $n$-dimensional Brownian motion. The action $E_T$ of the stochastic transformation on $(\mu,\sigma)$ is the
unique way of changing the SDE so that, if $(X,W)$ is a solution to $(\mu,\sigma)$, then $P_T(X,W)$ is a solution to $E_T(\mu,\sigma)$. \\
In this framework a symmetry is defined as a transformation $T$ which leaves the SDE $(\mu,\sigma)$ invariant. These transformations are the only ones
which preserve the set of solutions to the SDE $(\mu,\sigma)$. Since all actions $P_T$ and $E_T$ are explicitly determined in terms of $T=(\Phi,B, \eta)$,
it is possible to write the  determining equations satisfied by $T$ which can be solved explicitly with a
 computer algebra software (see \cite{DMU2}).\\
The main aim of the present paper is to generalize this approach  from Brownian-motion-driven SDEs to SDEs driven by general c\acc{a}dl\acc{a}g semimartingales taking values in (finite dimensional)  Lie
groups. There are two main differences with respect to the Brownian motion setting. The first one is the lack of a natural geometric transformation rule for processes with jumps replacing the It\^o transformation rule for continuous processes. This fact makes the action of a
diffeomorphism $\Phi$ on a SDE with such general driving noise more difficult to be described. The second one is the fact that a general semimartingale does not have the
symmetry properties of Brownian motion in the sense that we cannot \virgolette{rotate} it or make general time changes.\\

In order to address the first problem we restrict ourselves to a particular family of SDEs (that we call \emph{geometrical SDEs}) introduced by Cohen
in \cite{Cohen1996,Cohen1996(2)} (see also \cite{Applebaum1997,Cohen1995}). In particular, we consider SDEs defined by a map
$\Psi:M \times N \rightarrow M$, where $M$ is the (finite dimensional) manifold where the solution lives and $N$ is the (finite dimensional) Lie group where the driving process takes values. This
definition simplifies the description of the transformations of the solutions $(X,Z) \in M \times N$. In fact, if $(X,Z)$ is a solution to the SDE
$\Psi(x,z)$  then, for any diffeomorphism $\Phi$, $(\Phi(X),Z)$ is a solution to the SDE $\Phi(\Psi(\Phi^{-1}(x),z)$ (see Theorem
\ref{theorem_geometrical1} and Theorem \ref{theorem_gauge1}). We remark that the family of geometrical SDEs is not too restrictive: in fact it includes  affine
types SDEs, Marcus type SDEs, smooth SDEs driven by  L\'evy processes and even a class of iterated random maps (see Subsection \ref{subsection_comparison}
for further details).\\
The second problem is addressed by  using  two new notions of invariance of a semimartingale defined on a Lie group introduced in \cite{Annals}. These two notions are
extensions of predictable transformations which preserve the law of $n$ dimensional Brownian motion and $\alpha$-stable processes studied for
example in \cite[Chapter 4]{Kallenberg2005}. The first notion,  called \emph{gauge symmetry}, is a natural extension to semimartingales on Lie groups of the rotation invariance of
Brownian motion, while the second one, called \emph{time symmetry}, is a corresponding extension of the time rescaling invariance of Brownian motion.
The concept of gauge symmetry group is based on the action $\Xi_g$ of a Lie group $\mathcal{G}$ ($g$ is an element of $\mathcal{G}$) on the Lie
group $N$ which preserves the identity $1_N$ of $N$. A semimartingale $Z$ admits $\mathcal{G}$ as gauge symmetry group if, for any locally
bounded predictable process $G_t$ taking values in $\mathcal{G}$, the process $\tilde{Z}$, defined by the transformation $d\tilde{Z}=\Xi_{G_t}(dZ)$
has the same probability law as $Z$ (see Section
\ref{section_gauge}). A similar definition is given for the time symmetry, where $\Xi_g$ is replaced by an $\mathbb{R}_+$ action $\Gamma_r$ and the
process $G_t$ is replaced by an absolutely continuous time change $\beta_t$ (see Section \ref{section_time}).\\
Given an SDE $\Psi$ and a driving process $Z$ with gauge symmetry group $\Xi_g$ and time symmetry $\Gamma_r$, we are able to define a
stochastic transformation $T=(\Phi,B,\eta)$, where $\Phi$, $\eta$ are a diffeomorphism and a density of a time change as in the Brownian
setting, while $B$ is a function taking values in $\mathcal{G}$ (in the Brownian setting $\mathcal{G}$ is the group of rotation in
$\mathbb{R}^n$). In order to generalize \cite{DMU1}, using the properties of geometrical SDEs and of gauge and time symmetries, we define an action $E_T$ of $T$ on
the SDE $\Psi$ as well as an action $P_T$  of $T$ on the solutions $(X,Z)$.\\
In this setting we propose an algorithm for exploiting general symmetries in order to reduce a geometrical SDE admitting a solvable symmetry algebra to a more simple SDE and we provide
a reconstruction method to obtain solutions to the original SDE starting from the knowledge of the solutions to the reduced one.

\bigskip
Let us stress
two main  novelties of our approach. The first novelty is that, for the first time, the notion of symmetry of a SDE driven by 
general c\acc{a}dl\acc{a}g, in principle non-Markovian, semimartingales is introduced and studied in full detail. The analysis is based on the introduction of
a group of transformations which permits both the space transformation $\Phi$ and the gauge and time transformations $\Xi_g,\Gamma_r$. In this
way our approach extends the results of \cite{Cami2009}, where only general continuous semimartingales $Z$ and space transformations $\Phi$ are
considered. We also generalize the results to the case of  Markovian processes  on  $\mathbb{R}^m$ with  regular generators. Indeed, due to the
introduction of gauge and time symmetries, we recover most of the smooth symmetries of a Markovian process which would be lost if
we had just considered  the space transformation $\Phi$.\\

The second novelty of the paper is given by our explicit approach: indeed, we provide many results which permit to check explicitly whether a
semimartingale admits  given gauge and time symmetries and to compute stochastic transformations which are symmetries of a given SDE.
We obtain
the determining equations \refeqn{equation_determining2} which are satisfied by any infinitesimal symmetry, under some mild additional hypotheses on the jumps of the driving
process $Z$. The possibility of providing explicit determining equations is the main reason to restrict our
attention to geometrical SDEs instead of considering  more general classes of SDEs. Indeed, an interesting consequence of our study is that we provide a black-box method, applicable in several different situations,  which permits to explicitly compute symmetries of a given SDE or to construct
all the geometrical SDEs  admitting a given symmetry.
For these reasons, in order to show the generality and the
 user-friendliness of our theory, we conclude the paper with two applications.\\

The first one (see Section \ref{subsection_example}) is the study of symmetries of the general affine SDEs in two dimensions (see Section \ref{subsection_affine} for the definition of affine SDEs). This kind of model has many applications in the theory of iterated random maps (see, e.g., \cite{Arnold1998,Babillot1997,Kesten1973} and in particular \cite{Brockwell1991,Shumway2006} for the well known ARMA model) and it is closely connected with affine processes (see \cite{Filipovic2003} for the general definition of affine processes). \\
The second application (see Section  \ref{section_weak_numerical}) is the introduction of the general concept of weak symmetry for numerical schemes of SDEs and, in particular, for the Euler and Milstein numerical schemes for Brownian motion driven SDEs. This study is a generalization of \cite{DeVecchi2017}, where the concept of strong symmetries of Euler and Milstein numerical schemes was introduced. Section \ref{section_weak_numerical}, as well as \cite{DeVecchi2017}, are parts of a more general research line, aiming at exploiting Lie symmetries of ODEs and PDEs in order to obtain better numerical integrators
(see e.g. \cite{Winternitz2016,Dorodnitsyn2011,Levi2006,Levi2011} and references therein).\\

The paper is organized as follows. In Section \ref{section_geometrical} we introduce the notion of geometrical SDE and we discuss
their transformation properties under diffeomorphisms. In Section \ref{section_gauge} and in Section \ref{section_time} we
recall the definition of gauge and time symmetries with their main properties. In Section \ref{section_symmetry} we
extend the study of the symmetries of Brownian-motion-driven SDEs to SDEs driven by general c\acc{a}dl\acc{a}g semimartingales. Section \ref{subsection_example} is devoted to the detailed application of our symmetries analysis to a relevant example. In Section \ref{section_weak_numerical} we show that the general results of the paper can be successfully applied to the most common numerical approximations for Brownian motion driven SDEs.

\section{Stochastic differential equations with jumps}\label{section_geometrical}

In this section we introduce some preliminary material which will be fundamental in the rest of the paper.   Einstein's summation convention for repeated indices is used throughout the paper.

\subsection{Geometrical SDEs with jumps}

\begin{definition}
An adapted c\acc{a}dl\acc{a}g stochastic process $X$ on a subset $M$ of  $\mathbb{R}^m$ is a semimartingale if each component  can be decomposed as the sum of a local martingale and a  c\acc{a}dl\acc{a}g  adapted process whose sample paths are locally of bounded variation.
\end{definition}

Semimartingales are the largest class of processes for which an It\^o integral can be defined.
The importance of this class relies  on the fact that it is closed with respect to localization, change of time, absolutely continuous change of measure and optional stopping.
Therefore, working with this class of processes allows the use of the whole standard machinery of the stochastic calculus.
In particular, we recall the It\^o Lemma for (non continuous) semimartingales.

\begin{lemma}\label{lemma_Ito} \label{equation_Ito_formula}
If $X$ is a semimartingale on $M$ and $f$ is a twice continuously differentiable real-valued function on $M$, then $f(X)$ is a semimartingale and
\begin{eqnarray}
&f(X_t)=f(X_0) +\int_{0}^{t} \partial_{x^i} f(X_{s^-})dX_s^i+\frac{1}{2}\int_{0}^{t} \partial_{x^i,x^j} f(X_{s^-})d[X^i,X^j]^c_s\\
&+\sum_{0 \leq s \leq t}\{ (\Delta f(X_s))- \partial_{x^i} f(X_{s^-}) \Delta X_s^i \},&
\end{eqnarray}
where $\Delta X_s=X_s-X_{s^-}$ denotes the jump at time $s$ and $[X^i,X^j]^c$ is the continuous part of the quadratic variation defined by
$$[X]_t=[X]^c_t+X^2_0+\sum_{0 \leq s\leq t} (\Delta X_s)^2,$$
where $[X]_t:=[X,X]_t$.
\end{lemma}
\begin{proof}
The proof can be found in \cite{Protter1990} Chapter II, Section 7.
$\hfill$\end{proof}\\

Hereafter, denoting  by $N$ a matrix Lie group and fixing a linear representation of $N$, we denote  by $z^1, \ldots z^n$ its standard set of matrix elements.
For a semimartingale $Z$ on $N$,
a natural definition of jump process can be given.
Indeed, if $\tau$ is a stopping time, we define the jump at time
$\tau$ as the random variable $\Delta Z_{\tau}$ taking values on
$N$ such that
$$\Delta Z_{\tau}=Z_{\tau} \cdot (Z_{\tau_-})^{-1},$$
where $\cdot$ is the multiplication in the group $N$. In order to define a special class of equations that, in some sense,
depends only on the
jumps $\Delta Z_t$ of a process $Z$ defined on N, we consider  a topological space  $\mathcal{K}$ and we introduce a function $\Psi$ of the form
$$\Psi_{\cdot}(\cdot,\cdot):M \times N \times \mathcal{K} \rightarrow M,$$
such that $\Psi_k(x,1_N)=x$ for any $k\in \mathcal{K}$ and $x \in M$, ${\Psi}_k$ is smooth in the $M, N$ variables and ${\Psi}_k$ and all its derivatives with respect to the $M,N$ variables are continuous in all their arguments. In order to provide a general concept of stochastic differential equation (SDE) with jumps defined on $M$ and driven by a general c\acc{a}dl\acc{a}g
semimartingale  on  $N$  we introduce an auxiliary function $\overline{\Psi}_k$
\begin{equation}
\overline{\Psi}_.(.,.,.): M \times N \times N \times \mathcal{K}\rightarrow M
\end{equation}
defined as
\begin{equation}\label{equation_manifold0}
\overline{\Psi}_k(x,z',z)=\Psi_k(x,z'\cdot z^{-1})={\Psi}_k(x,\Delta z).
\end{equation}

If $K$ is a predictable process such that there exist an increasing sequence of stopping times $\tau_n\to +\infty$ and an increasing sequence of compact sets $\mathfrak{K}_n\subset \mathcal{K}$ such that $K_t(\omega)\in \mathfrak{K}_n$ when $0<t\le\tau(\omega)$, we say that $K$ is a \emph{predictable and locally bounded process}.
In this setting,  inspired by  \cite{Cohen1996}, we say that the semimartingale $X$ in $M$ is a solution to the SDE defined by ${\Psi}_{K_t}$ and driven by the semimartingale $Z$ on $N$ (and we write $dX_t= {\Psi}_{K_t}(dZ_t)$) if and only if

\begin{equation}\label{equation_manifold2}
\begin{array}{ccl}
X^i_t-X^i_0&=&\int_0^t{\partial_{z'^{\alpha}}(\overline{\Psi}^i_{K_s})
(X_{s_-},Z_{s_-},Z_{s_-})dZ^{\alpha}_s}+\frac 12 \int_0^t{\partial_{z'^{\alpha},z'^{\beta}}(\overline{\Psi}^i_{K_s})(X_{s_-},Z_{s_-},Z_{s_-})d[Z^{\alpha},Z^{\beta}]_s}\\
&+&\sum_{0\leq s \leq
t}\{\overline{\Psi}^i_{K_s}(X_{s_-},Z_{s},Z_{s_-})-\overline{\Psi}^i_{K_s}(X_{s_-},Z_{s_-},Z_{s_-})-\partial_{z'^{\alpha}}(\overline{\Psi}^i_{K_s})(X_{s_-},Z_{s_-},Z_{s_-})\Delta
Z^{\alpha}_s\}.
\end{array}
\end{equation}

The previous discussion can be resumed in the following

\begin{definition}\label{definition_solution}
Let $K$ be a predictable locally
bounded process taking values in $\mathcal{K}$. A pair of semimartingales $(X,Z)$ on $M$ and $N$ respectively is a solution to the \emph{geometrical
SDE} defined by ${\Psi}_{K_t}$  if $X$ and $Z$ solve the integral equation
\refeqn{equation_manifold2} (with $\overline{\Psi}$ given by \eqref{equation_manifold0}).
When we consider $\mathcal{K}$ consisting of a single point $\{k_0\}$, we write $\Psi$ instead of $\Psi_{k_0}$.
\end{definition}

\begin{theorem}\label{theorem_manifold1}
For any semimartingale $Z$ on
$N$, there exist a stopping time $\tau$ and a semimartingale $X$
on $M$ such that $(X^{\tau}, Z^{\tau})$ is a solution to \eqref{equation_manifold2}
\end{theorem}
\begin{proof}
The proof can be found in \cite{Cohen1996},
Theorem 2.
${}\hfill$ \end{proof}\\

\subsection{Geometrical SDEs and diffeomorphisms}

The notion of geometrical SDE introduced  in Definition \ref{definition_solution} naturally suggests to consider transformations of solutions to a SDE.

\begin{theorem}\label{theorem_geometrical1}
Given $M,M'\subset \mathbb{R}^m$ and $N,N'$ two matrix Lie groups,  let $\Phi:M \rightarrow M'$ and $\tilde{\Phi}:N \rightarrow N'$ be two diffeomorphisms which respect the group action. If $(X,Z)$ is a solution to the geometrical SDE
${\Psi}_{K_t}$, then $(\Phi(X),\tilde{\Phi}(Z))$ is a solution to the geometrical SDE ${\Psi}'_{K_t}$ defined  by
$$\overline{\Psi}'_{K_t}(x,z',z)=\Phi(\overline{\Psi}_{K_t}(\Phi^{-1}(x),\tilde{\Phi}^{-1}(z'),\tilde{\Phi}^{-1}(z))).$$
\end{theorem}

In order to prove Theorem \ref{theorem_geometrical1} we recall the following  lemma.

\begin{lemma}\label{lemma_geometrical1}
Given $k$  c\acc{a}dl\acc{a}g semimartingales  $X^1,...,X^k$, let $H^{\alpha}_1,...,H^{\alpha}_k$, for $\alpha=1,...,r$, be predictable processes
which can be integrated along $X^1,...,X^k$ respectively. If $\Phi^{\alpha}(t,\omega,x^1,x'^1,...,x^k,x'^k):\mathbb{R}_+ \times \Omega
\times \mathbb{R}^{2k} \rightarrow \mathbb{R}$ are some progressively measurable random functions continuous in
$x^1,x'^1,...,x^k,x'^k$ and such that $|\Phi^{\alpha}(t,\omega,x^1,x'^1,...,x^k,x'^k)| \leq O((x^1-x'^1)^2+...+(x^k-x'^k)^2)$ as $x^i
\rightarrow x'^i$, for  almost every fixed $\omega \in \Omega$ and uniformly on compact subsets of $\mathbb{R}_+ \times \mathbb{R}^{2k}$, the processes
$$Z^{\alpha}_t=\int_0^t{H^{\alpha}_{i,s}dX^i_s}+\sum_{0 \leq s \leq t}\Phi^{\alpha}(s,\omega,X^1_{s_-},X^1_s,...,X^k_{s_-},X^k_s)$$
are semimartingales. Furthermore
\begin{eqnarray}
&\Delta Z^{\alpha}_t=H^{\alpha}_{i,t} \Delta X^i_t + \Phi^{\alpha}(t,\omega,X^1_{t_-},X^1_t,...,X^k_{t_-},X^k_t),\label{equation_gauge4}\\
&[Z^{\alpha},Z^{\beta}]^c_t=\int_0^t{H^{\alpha}_{i,s}H^{\beta}_{j,s}d[X^i,X^j]^c_s},&\label{equation_gauge2}\\
&\int_0^t{K_{\alpha,s}dZ^{\alpha}_s}=\int_0^t{K_{\alpha,s}H^{\alpha}_{i,s}dX^i_s}+\sum_{0 \leq s \leq t}K_{\alpha,s}\Phi^{\alpha}(s,\omega,X^1_{s_-},X^1_s,....,X^k_{s_-},X^k_s),&\label{equation_gauge3}
\end{eqnarray}
where $K_{\alpha,s}$ is any predictable locally bounded process.
\end{lemma}
\begin{proof}
The proof can be found in \cite{Annals}, Lemma 2.10.
\hfill\end{proof}\\

\begin{remark}\label{remark_geometrical1}
Let $ \mathcal{K}$ be a metric space, $K \in \mathcal{K}$ be a locally bounded predictable process and $\tilde{\Phi}: \mathbb{R}_+ \times
\mathcal{K} \times \mathbb{R}^{2k} \rightarrow \mathbb{R}$ be a $C^2$ function in $\mathbb{R}_+$ and depending on $\mathbb{R}^{2k}$ variables such that
$\tilde{\Phi}$ and all its derivatives are continuous in all of  their arguments. If
$\tilde{\Phi}(\cdot,\cdot,x^1,x^1,...,x^k,x^k)=\partial_{x'^i}(\tilde{\Phi})(\cdot,\cdot,x^1,x^1,...,x^k,x^k)=0$
 for $i=1,...k$, then $\Phi(t,\omega,...)=\tilde{\Phi}(t,K_t(\omega),...)$ satisfies the
hypothesis of Lemma \ref{lemma_geometrical1}.
\end{remark}

\begin{proof}[Proof of Theorem \ref{theorem_geometrical1}]

In order to simplify the proof we consider the two special cases $M=M'$, $\Phi=Id_M$ and $N=N'$, $\tilde{\Phi}=Id_N$.
The general case can be obtained combining  these two cases.\\
If $M=M'$ and $\Phi=Id_M$,  putting
$\tilde{Z}=\tilde{\Phi}(Z)$, so that  $Z=\tilde{\Phi}^{-1}(\tilde {Z})$,
by It\^o formula for semimartingales with jumps, Lemma
\ref{lemma_geometrical1} and Remark \ref{remark_geometrical1} we
have
\begin{eqnarray*}
Z^{\alpha}_t-Z^{\alpha}_0&=&\int_0^t{\partial_{\tilde{z}^{\beta}}(\tilde{\Phi}^{-1})^{\alpha}(\tilde{Z}_{s_-})d\tilde{Z}^{\beta}_s+\partial_{\tilde{z}^{\beta}\tilde{z}^{\gamma}}(\tilde{\Phi}^{-1})^{\alpha}(\tilde{Z}_{s_-}) d[\tilde{Z}^{\beta},\tilde{Z}^{\gamma}]^c_s}+\\
&&+\sum_{0\leq s \leq t}((\tilde{\Phi}^{-1})^{\alpha}(\tilde{Z}_s)-(\tilde{\Phi}^{-1})^{\alpha}(\tilde{Z}_{s_-})-\partial_{\tilde{z}^{\beta}}(\tilde{\Phi}^{-1})^{\alpha}(\tilde{Z}_{s_-})\Delta\tilde{Z}^{\beta}_s\\
d[Z^{\alpha},Z^{\beta}]^c_t&=&\partial_{\tilde{z}^{\gamma}}(\tilde{\Phi}^{-1})^{\alpha}(Z_{s_-})\partial_{\tilde{z}^{\delta}}(\tilde{\Phi}^{-1})^{\beta}(Z_{s_-})d[\tilde{Z}^{\gamma},\tilde{Z}^{\delta}]^c_t\\
\Delta
Z^{\alpha}_t&=&(\tilde{\Phi}^{-1})^{\alpha}(\tilde{Z}_t)-(\tilde{\Phi}^{-1})^{\alpha}(\tilde{Z}_{t_-}).
\end{eqnarray*}
The conclusion of  Theorem \ref{theorem_geometrical1} follows using the definition of solution to the geometrical SDE $\Psi_{K_t}$, Lemma \ref{lemma_geometrical1} and the chain rule
for derivatives.\\
Suppose now that $N=N'$ and $\tilde{\Phi}=Id_N$. Putting $X'=\Phi(X)$,
by It\^o formula we obtain
\begin{eqnarray*}
dX'^i_t&=& {\partial_{x^j}(\Phi^i)(X_{s_-})dX^j_s}+{\partial_{x^jx^h}(\Phi^i)(X_{s_-})d[X^j,X^h]^c_s}+\\
&&+(\Phi^i(X_s)-\Phi^i(X_{s_-})-\partial_{x^j}(\Phi^i)(X_{s_-})\Delta
X^j_s).
\end{eqnarray*}
Furthermore, by definition of solution to the geometrical SDE
$\Psi$ and by  Lemma \ref{lemma_geometrical1} we have
\begin{eqnarray*}
dX^i_s&=&\partial_{z'^{\alpha}}(\overline{\Psi}^i_{K_s})(X_{s_-},Z_{s_-},Z_{s_-})
dZ^{\alpha}_s+\partial_{z'^{\alpha},z'^{\beta}}(\overline{\Psi}_{K_s}^i)(X_{s_-},Z_{s_-},Z_{s_-})d[Z^{\alpha},Z^{\beta}]^c_s +\\
&&+\overline{\Psi}_{K_s}^i(X_{s_-},Z_{s},Z_{s_-})-\overline{\Psi}_{K_s}^i(X_{s_-},Z_{s_-},Z_{s_-})-
\partial_{z'^{\alpha}}(\overline{\Psi}_{K_s}^i)(X_{s_-},Z_{s_-},Z_{s_-})\Delta Z^{\alpha}_s,\\
d[X^i,X^j]^c_s&=&\partial_{z'^{\alpha}}(\overline{\Psi}_{K_s}^i)(X_{s_-},Z_{s_-},Z_{s_-})
\partial_{z'^{\beta}}(\overline{\Psi}_{K_s}^i)(X_{s_-},Z_{s_-},Z_{s_-})d[Z^{\alpha},Z^{\beta}]^c_s\\
\Delta X^i_s&=&\overline{\Psi}_{K_s}^i(X_{s_-},Z_s,Z_{s_-})-\overline{\Psi}_{K_s}^i(X_{s_-},Z_{s_-},Z_{s_-}).
\end{eqnarray*}
Using the previous relations, the fact that $X=\Phi^{-1}(X')$ and the chain rule for derivatives we get the thesis. ${}\hfill$ \end{proof}

\subsection{A comparison with other approaches}\label{subsection_comparison}

In order to show the generality of our approach (mainly inspired by \cite{Cohen1996}) in this section we show how
the definition of geometrical SDEs driven by semimartingales with jumps may include many interesting classes  of SDEs
driven by c\acc{a}dl\acc{a}g processes appearing in the literature such as
 affine-type SDEs  of the type studied in \cite[Chapter
V]{Protter1990} and \cite[Chapter 5]{Bichteler2002},  as well as SDEs driven
by L\'evy processes with smooth coefficients (see e.g. \cite{Applebaum2004,Kunita2004}) and smooth iterated random functions (see e.g.
\cite{Arnold1998,Diaconis1999}). Moreover, Marcus-type SDEs (see \cite{Protter1995,Marcus1978,Marcus1981}) can be successfully considered in this setting: the
interested reader is referred to \cite{Annals} for a detailed study of this topic.

\subsubsection{Affine-type SDEs}\label{subsection_affine}

We briefly describe the affine type SDEs as proposed e.g. in \cite[Chapter V]{Protter1990}. In particular we show how it is possible to rewrite them according to our geometrical setting.\\
Let $(Z^1,...,Z^n)$ be a semimartingale in $N$ and let $\sigma:M
\rightarrow \matr(m,n)$ be a smooth function taking values in the
set of $m \times n$ matrices with real elements. We consider the SDE defined by
\begin{equation}\label{equation_affine}
dX^i_t=\sigma^i_{\alpha}(X_t) dZ^{\alpha}_t,
\end{equation}
where $\sigma^i_j$ are the components of the matrix $\sigma$. If $Z^1_t=t$  and $Z^2,...,Z^n$ are independent Brownian motions, we have the usual diffusion processes with drift $(\sigma^1_1,...,\sigma^m_1)$ and diffusion matrix $(\sigma^i_{\alpha})|_{\stackrel{i=1,...,m}{\alpha=2,...,n}}$.\\
The previous affine-type SDE can be rewritten as a geometrical SDE
defined by the function $\overline{\Psi}$
$$\overline{\Psi}(x,z',z)=x+\sigma(x) \cdot (z'-z),$$
or, in coordinates,
$$\overline{\Psi}^i(x,z',z)=x^i+\sigma^i_{\alpha}(x)(z'^{\alpha}-z^{\alpha}).$$
In fact, by definition of geometrical SDE
$\Psi$, we have
\begin{eqnarray*}
X^i_t-X^i_0&=&\int_0^t{\partial_{z'^{\alpha}}(\overline{\Psi}^i)(X_{s_-},Z_{s_-},Z_{s_-})dZ^{\alpha}_s}+\int_0^t{\partial_{z'^{\alpha},z'^{\beta}}(\overline{\Psi}^i)(X_{s_-},Z_{s_-},Z_{s_-})d[Z^{\alpha},Z^{\beta}]_s}\\
&+&\sum_{0\leq s \leq t}\{\overline{\Psi}^i(X_{s_-},Z_{s},Z_{s_-})-\overline{\Psi}^i(X_{s_-},Z_{s_-},Z_{s_-})-\partial_{z'^{\alpha}}(\overline{\Psi}^i)(X_{s_-},Z_{s_-},Z_{s_-})\Delta Z^{\alpha}_s\},\\
&=&\int_0^t{\sigma^i_{\alpha}(X_{s_-})dZ^{\alpha}_s}+\sum_{0\leq s \leq t}\{\sigma^i_{\alpha}(X_{s_-})(Z^{\alpha}_s-Z^{\alpha}_{s_-})-\sigma^i_{\alpha}(X_{s_-})\Delta Z^{\alpha}_s\}\\
&=&\int_0^t{\sigma^i_{\alpha}(X_{s_-})dZ^{\alpha}_s}.
\end{eqnarray*}

\subsubsection{Smooth SDEs driven by a L\'evy process}\label{subsubsection_Levy}

In this section we introduce a class of SDEs driven by $\mathbb{R}^n$-valued L\'evy processes (see, e.g.,
\cite{Applebaum2004,Kunita2004}), which we denote as smooth SDEs. It is well known that an $\mathbb{R}^n$-valued L\'evy process $(Z^1,...,Z^n)$ can be always decomposed into the sum of suitable Brownian
motions and compensated Poisson processes defined on $\mathbb{R}^n$. Moreover such process can be identified by a
vector $b_0=(b_0^1,...,b_0^n) \in \mathbb{R}^n$, an $n \times n$ matrix $A_0^{\alpha\beta}$ and a positive $\sigma$-finite measure $\nu_0$
defined on $\mathbb{R}^n$ (called L\'evy measure, see, e.g., \cite{Applebaum2004,Sato1999}) such that
$$\int_{\mathbb{R}^n}{\frac{|z|^2}{1+|z|^2}d\nu_0(z)} < + \infty.$$
By  L\'evy-It\^o decomposition, the triplet $(b_0,A_0,\nu_0)$ is such that there exist an $n$ dimensional Brownian motion $(W^1,...,W^n)$ and a
Poisson measure $P(dz,dt)$ defined on $\mathbb{R}^n$ such that
\begin{eqnarray*}
Z^{\alpha}_t&=&b^{\alpha}_0 t+ C^{\alpha}_{\beta} W^{\beta}_t+\int_0^t{\int_{|z| \leq 1}{z^{\alpha}(P(dz,ds)-d\nu_0(z)ds)}}+\\
&&+\int_0^t{\int_{|z|>1}{z^{\alpha}P(dz,ds)}},
\end{eqnarray*}
where $A^{\alpha\beta}_0=\frac{1}{2}\sum_{\gamma}C^{\alpha}_{\gamma}C^{\beta}_{\gamma}$. We assume for simplicity that $b_0^1=1$ and
$b^{\alpha}_0=0$ for $\alpha>1$, that there exists
 $n_1$ such that $A^{\alpha\beta}_0=\delta^{\alpha\beta}$ for $1< \alpha ,\beta \leq n_1$ and
 $A^{\alpha\beta}_0=0$ for $\alpha$ or $\beta$ in $\{1,n_1+1,...,n\}$, and finally  that
 $\int_0^t{\int_{|z| \leq 1}{z^{\alpha}(P(dz,ds)-d\nu_0(z)ds}}=0$ and $\int_0^t{\int_{|z|>1}{z^{\alpha}P(dz,ds)}}=0$ for $\alpha \leq n_1$. \\
Given a vector field $\mu$ on $M$, a set of $n_1-1$ vector
fields $\sigma=(\sigma_2,...,\sigma_{n_{1}})$ on $M$ and a smooth
(both in $x $ and $z$) function $F:M \times \mathbb{R}^{n-n_1}
\rightarrow \mathbb{R}^m$ such that $F(x,0)=0$, we  say that
a semimartingale $X \in M$ is a solution to the smooth SDE
$(\mu,\sigma,F)$ driven by the $\mathbb{R}^n$-valued L\'evy process
$(Z^1,...,Z^n)$ if
\begin{eqnarray*}
X^i_t-X^i_0&=&\int_0^t{\mu^i(X_{s_-})dZ^1_s}+\int_0^t{\sum_{\alpha=2}^{n_1}\sigma^i_{\alpha}(X_{s_-})dZ_s^{\alpha}}+\\
&&+\int_0^t{\int_{\mathbb{R}^{n-n_1}}{F^i(X_{s_-},z)(P(dz,ds)-I_{|z|\leq 1}\nu_0(dy)ds)}},
\end{eqnarray*}
with $I_{|z|\leq 1}$ the indicator function of the set
$\{|z|\leq 1\}\subset \mathbb{R}^{n-n_1}$. We introduce the function
$$\overline{\Psi}^i(x,z',z)=x^i+\tilde{\mu}^i(x)(z'^1-z^1)+\sigma^i_{\alpha}(x)(z'^{\alpha}-z^{\alpha})+F^i(x,z'-z),$$
where
$$\tilde{\mu}^i(x)=\mu^i(x)-\int_{|z|\leq 1}{(F^i(x,z)-\partial_{z^{\alpha}}(F^i)(x,z)z^{\alpha})d\nu_0(z)}.$$
Then any solution $X$ to the smooth SDE
$(\mu,\sigma,F)$ driven by the L\'evy process $(Z^1,...,Z^n)$ is
 solution to the geometrical SDE associated with $\overline{\Psi}$ and driven by
the $\mathbb{R}^n$ semimartingale $(Z^1,...,Z^n)$. Also the converse is true.

\begin{remark}
Our request on smoothness of $F$ in both $x,z$ is  a stronger requirement with respect to the usual assumption
that $F$ is Lipschitz in $x$ and measurable in $z$. This stronger regularity assumption is the reason for the denomination  \emph{smooth SDE} driven by a L\'evy process.
\end{remark}

\subsubsection{Iterated random smooth functions}\label{subsection_iterated}

In the previous sections we have only  considered continuous time processes $Z_t, t \in \mathbb{R}_{+}$. Let us now take $Z$ as a discrete time
adapted process, i.e. $Z$ is a sequence of random variables $Z_0,Z_1,...,Z_l,...$ defined on $N$. We can consider $Z$ as a c\acc{a}dl\acc{a}g
continuous time process $Z_t$ defined as
$$Z_t=Z_l \text{ if } l \leq t < l+1, \quad  t \in \mathbb{R}_{+}.$$
Since the process $Z$ is a pure jump process with a finite number of jumps in any compact interval of $\mathbb{R}_+$, $Z$ is a
semimartingale. If $(X,Z)$ is a solution to the geometrical SDE  ${\Psi}$, we have that
\begin{equation}\label{equation_difference}
X_l={\Psi}(X_{l-1},\Delta Z_l)
\end{equation}
and $X_t=X_l$ if $l \leq t < l+1$. The process $X$ can be viewed as a discrete time process defined by the recursive relation
\refeqn{equation_difference}. These processes are special forms of \emph{iterated random functions} (see e.g.
\cite{Arnold1998,Diaconis1999,Schreiber2012}) and  this kind of equations is very important in time series analysis (see, e.g.,
\cite{Brockwell1991,Shumway2006}) and in numerical simulation of SDEs (see, e.g. \cite{Kloeden1992} for simulation of SDEs and
\cite{DeVecchi2017} for the concept of strong symmetry of a discretization scheme). In this case we do not need that ${\Psi}$ is smooth
in all its variables and that ${\Psi}(x,1_N)=x$ for any $x \in M$. In the case of a discrete time semimartingale $Z_t$
these two conditions can be skipped and
we can consider more general iterated random functions defined by relation \refeqn{equation_difference}.\\
An important example of iterated random functions can be obtained by considering $M=\mathbb{R}^m$, $N=GL(m) \times
\mathbb{R}^m$ and the functions
$$\overline{\Psi}(x,z',z)=(z'_1 \cdot z_1^{-1}) \cdot x + (z'_2-z_2),$$
where $(z_1,z_2) \in GL(m) \times \mathbb{R}^m$. Moreover, taking  two sequences of random variables $A_0,...,A_l,... \in GL(n)$ and
$B_0,...,B_l,... \in \mathbb{R}^m$, we  define
$$Z_l=\left(A_l \cdot A_{l-1} \cdot ....\cdot A_0, B_l +B_{l-1} + ....+ B_0  \right).$$
The iterated random functions associated with the SDE ${\Psi}$ is
$$X_l=A_l \cdot X_{l-1}+B_{l}.$$
This model is very well studied (see, e.g. \cite{Arnold1998,Babillot1997,Kesten1973}). In particular the well known  ARMA model is of this form
(see, e.g., \cite{Brockwell1991,Shumway2006}).

\section{Gauge transformations and gauge symmetries}\label{section_gauge}

In this section, we  generalize the  well known noise change property of affine-type SDEs driven by c\acc{a}dl\acc{a}g semimartingales.
If $M=\mathbb{R}^m$ and $N \subseteq \mathbb{R}^n$, we can consider the
affine SDE
$$dX^i_t=\sigma^i_{\alpha}(X_{t_-}) dZ^{\alpha}_t,$$
and we can define a new semimartingale on $N$ given by
\begin{equation}\label{equation_semimartingales_change1}
d\tilde{Z}^{\alpha}_t=B^{\alpha}_{\beta,t}dZ^{\beta}_t,
\end{equation}
where $B=(B^{\alpha}_{\beta})$ is a locally bounded predictable
process taking values in $GL(n)$. Therefore, the affine SDE can be rewritten in terms of the
semimartingale $\tilde{Z}$ in the following way
\begin{equation}\label{equation_property}
dX^i_t=\sigma^{i}_{\alpha}(X_{t_-})(B^{-1})^{\alpha}_{\beta,t} d
\tilde{Z}^{\beta}_t,
\end{equation}
where $B^{-1}$ is the inverse matrix of $B$.\\

The notion of geometrical SDEs allows us to introduce  a useful generalization of the semimartingales change rule \refeqn{equation_property}.\\
Given a Lie group $\mathcal{G}$,  suppose that $M=\tilde{N}$ for some Lie group $\tilde{N}$. If we
consider a smooth function
$$\Xi_{\cdot}(\cdot):N \times \mathcal{G} \rightarrow \tilde{N}$$
satisfying  $\Xi_g(1_N)=1_{\tilde{N}}, \forall g \in \mathcal{G}$, we can define the map
$${\Psi}^{\Xi}_g(x,z)=\Xi_g(z) \cdot x.$$
If $Z$ is a semimartingale on $N$,  we define the transformed semimartingale on $\tilde{N}$ by
\begin{equation}\label{equation_semimartingales_change2}
d\tilde{Z_t}=\Xi_{G_t}(dZ_t)
\end{equation}
as the unique solution $(\tilde{Z},Z)$  to the equation
$$d\tilde{Z}_t={\Psi}^{\Xi}_{G_t}(d Z_t),$$
with initial condition $\tilde{Z}_0=1_{\tilde{N}}$. Before proving further results about  transformation \refeqn{equation_semimartingales_change2}, we
 show that the semimartingales change \refeqn{equation_semimartingales_change1} is a particular case of
\refeqn{equation_semimartingales_change2}. In fact, for $\tilde{N}=N=\mathbb{R}^n$, any map $\Xi_{\cdot}:\mathbb{R}^n \times \mathcal{G}
\rightarrow \mathbb{R}^n$ gives the geometrical SDE defined by the function
$${\Psi}^{\Xi}_g(\tilde{z},z)=\tilde{z}+\Xi_g(z).$$
This means that equation \refeqn{equation_semimartingales_change2} is explicitly given by the relation
\begin{equation}\label{equation_semimartingales_change3}
\begin{array}{ccl}
\tilde{Z}_t&=&\int_0^t{\partial_{z^{\alpha}}(\Xi_{G_s})(0)dZ^{\alpha}_s}+\int_0^t{\partial_{z^{\alpha}z^{\beta}}(\Xi_{G_s})(0)d[Z^{\alpha},Z^{\beta}]_s^c}+\\
&&+\sum_{0\leq s \leq t}(\Xi_{G_s}(\Delta
Z_s)-\partial_{z^{\alpha}}(\Xi_{G_s})(0)\Delta Z^{\alpha}_s).
\end{array}
\end{equation}
If $\mathcal{G}=GL(n)$ and $\Xi_g(z)=\Xi_B(z)=B \cdot z$, since both $\partial_{z^{\alpha}z^{\beta}}(\Xi_{B_t})(0)$ and $(\Xi_{G_s}(\Delta
Z_s)-\partial_{z^{\alpha}}(\Xi_{G_s})(0)\Delta Z^{\alpha}_s)$ are equal to zero, we obtain equation \refeqn{equation_semimartingales_change1}.

\begin{theorem}\label{theorem_gauge1}
Let $N,\tilde{N}$ be two Lie groups and suppose that
$(X,\tilde{Z})$ (where $\tilde{Z}$ is defined on $\tilde{N}$) is a
solution to the geometrical SDE $\Psi_{K_t}$. If
$d\tilde{Z}_t=\Xi_{G_t}(dZ_t)$,  then $(X,Z)$ is a solution to the geometrical
 SDE defined by
$$\hat{\Psi}_{k,g}(x,z)=\Psi_k(x,\Xi_g(z)).$$
\end{theorem}
\begin{proof}
For the proof see Theorem 2.8 in \cite{Annals}.
${}\hfill$ \end{proof}
\begin{corollary}\label{corollary_gauge1}
Suppose that $\mathcal{G}$ is a Lie group and  $\Xi$
is a Lie group action. If $(X,Z)$ is a solution to the
geometrical SDE $\Psi_{K_t}$, then $(X,\tilde{Z})$ is a solution to
the geometrical SDE defined by
$$\hat{\Psi}_{k,g}(x,z)=\Psi_k(x,\Xi_{g^{-1}}(z)).$$
\end{corollary}
\begin{proof}
The proof is an application of Theorem \ref{theorem_gauge1} and of the fact that $dZ_t=\Xi_{G_t^{-1}}(d\tilde{Z}_t)$. Indeed, defining
$d\hat{Z}_t=\Xi_{G_t^{-1}}(d\tilde{Z}_t)$, by  Theorem \ref{theorem_gauge1} we have that $d\hat{Z}_t=\Xi_{G_t^{-1}} \circ \Xi_{G_t}
(dZ_t)=\Xi_{1_{\mathcal{G}}}(dZ_t)=dZ_t$. The corollary follows directly from Theorem \ref{theorem_gauge1}. ${}\hfill$ \end{proof}

\subsection{Definition of gauge symmetries}

Let us consider the following well known
property of  Brownian motion. Let $B_t: \Omega \times [0,T]
\rightarrow SO(n)$ be a predictable process taking values in the
Lie group of special orthogonal matrices and consider a Brownian
motion $Z$ on $\mathbb{R}^n$. Then the process defined by
\begin{equation}\label{equation_gauge1}
Z'^{\alpha}_t=\int_0^t{B^{\alpha}_{\beta,s}dZ^{\beta}_s}
\end{equation}
is a new $n$ dimensional Brownian motion.\\
In \cite{Annals}  the  generalization of this property to the case in which $Z$ is a c\acc{a}dl\acc{a}g semimartingale in a Lie group $N$ has been proposed (see
\cite{Privault2012} for a similar result about Poisson measures). In the simple case $N=\mathbb{R}^n$,  replacing the Brownian motion by a
general semimartingale, the invariance property \refeqn{equation_gauge1}  is no longer true.

\begin{definition}\label{definition_gauge}
Let $Z$ be a semimartingale on a Lie group $N$ with respect to the filtration $\mathcal{F}_t$. Given a Lie group $\mathcal{G}$  and $ g\in \mathcal{G}$, we say that $Z$ admits $\mathcal{G}$, with action $\Xi_g$ and with respect to the filtration $\mathcal{F}_t$,  as \emph{gauge symmetry group} if, for any $\mathcal{F}_t$-predictable locally bounded process $G_t$ taking values in $\mathcal{G}$, the semimartingale $\tilde{Z}$ solution to the equation $d\tilde{Z}_t=\Xi_{G_t}(dZ_t)$ has the same law as $Z$.
\end{definition}

In the following we consider the filtration $\mathcal{F}_t$ of the probability space $(\Omega,\mathcal{F},\mathbb{P})$ as given and we do not mention it when it is  not strictly necessary.\\
Since a necessary condition for Definition \ref{definition_gauge} is that the solution  $\tilde{Z}$ to the  equation
$d\tilde{Z}_t=\Xi_{G_t}(dZ_t)$, is defined for all times, we are interested in characterizing  SDEs of the previous form with explosion time equal to $+ \infty$ for any $G_t$.
The following proposition provides a sufficient condition on the group $N$ so that, for any action $\Xi_g$, the corresponding geometrical  SDE has indeed
explosion time $+ \infty$.

In \cite{Annals}, in order to provide some explicit methods for checking when a semimartingale on a Lie group $N$ admits $\mathcal{G}$, with action $\Xi_g$, as gauge symmetry group,
the concept of characteristics of a semimartingale on a Lie group has been introduced. This allows to formulate  a condition, equivalent to Definition \ref{definition_gauge}, that can be
directly applied to L\'evy processes on Lie groups providing a completely deterministic method  to verify Definition \ref{definition_gauge} in this
case. Moreover, one can use this reformulation  to provide some examples of non-Markovian processes admitting  gauge
symmetry groups (see \cite{Annals}).\\

\subsection{Gauge symmetries of  L\'evy processes}\label{subsection_gauge_Levy}

In this section we recall the general theory of gauge symmetries of L\'evy processes on Lie groups as described in \cite{Annals}.
In order to fix notations, given  $n$ generators $Y_1,...,Y_n$ of right-invariant vector fields on $N$, we introduce a set of functions $h^1,...,h^n$ (called \emph{truncated
functions related to $Y_1,...,Y_n$}) which are measurable, bounded, smooth in a neighborhood of
the identity $1_N$, with compact support and such that $h^{\alpha}(1_N)=0$ and $Y_{\alpha}(h^{\beta})(1_N)=\delta^{\beta}_{\alpha}$ (the existence of
these functions is proved, for example, in
\cite{Hunt1956} and they can be chosen to be equal to a set of canonical coordinates in a neighborhood of $1_N$).

\begin{definition}\label{definition_Levy}
Let $A_0$ be an $n\times n$
 symmetric non-negative matrix, $b_0$ be an $n$-dimensional vector and $\nu_0$ be a positive measure such that $\int_N{(h^{\alpha}(z))^2\nu_0(dz)}<+\infty$ and
$\int_N{f(z)\nu_0(dz)}<+\infty $ for any smooth and bounded function $f \in \cinf(N)$ which is identically zero in a neighborhood of $1_N$.
A c\acc{a}dl\acc{a}g semimartingale $Z$ on a Lie group $N$ is called a L\'evy process with characteristics $(b_0,A_0,\nu_0)$ if, for any $C^2$ function $f$ defined on
 $N$ and with compact support, we have that
\begin{multline*}
f(Z_t)-\int_0^t{L(f)(Z_{s_-})ds}:=\\f(Z_t)-\left( b_0^{\alpha}Y_{\alpha}(f)(z)+A^{\alpha\beta}_0
Y_{\alpha}(Y_{\beta}(f))(z)
+\int_N{(f(z' \cdot
z^{-1})-f(z)-h^{\alpha}(z')Y_{\alpha}(f)(z))\nu_0(dz')}\right)
\end{multline*}
is a local martingales.
\end{definition}

\begin{remark}
The characteristics of a  L\'evy process introduced in Definition \ref{definition_Levy} are the same as those of Subsection \ref{subsubsection_Levy}.
Furthermore if $Z$ is a L\'evy process, then $Z$ is also an homogeneous Markov process. Indeed the operator $L$ with domain given by the set of smooth functions with compact support is the generator of the L\'evy process $Z$.
\end{remark}

Since the map $\Xi_g$ is such that $\Xi_g(1_N)=1_N$, using the identification of the Lie algebra  generated by the right-invariant vector
fields $Y_1,...,Y_n$ with $\mathfrak{n}=T_{1_N}N$, we can define a linear maps $\Upsilon_g:\mathfrak{n} \rightarrow \mathfrak{n}$  in the following way
\begin{eqnarray*}
\Upsilon_g(Y)&=&\partial_a(\Xi_g \circ \Phi_a)(1_N)|_{a=0},
\end{eqnarray*}
where $Y$ is an invariant vector field with associated flow $\Phi_a:N
\rightarrow N$. Furthermore, in order to simplify the following treatment, hereafter we require that, for any pair of flows $\Phi_a, \Phi'_b$
$$ \partial_b(\partial_a(\Xi_g \circ \Phi_a \circ \Phi'_b))(1_N)|_{a=0,b=0}=0. $$
This condition (which is always satisfied for the linear actions considered in this paper) can be relaxed: see \cite{Annals} for the general theory.\\
Since $Y(f)(z)=\partial_a(f\circ \Phi_a)(z)|_{a=0}$ and the flow of right-invariant vector fields commutes with the right multiplication, we can exploit
the definition of $\Psi^{\Xi}_g(\tilde{z},z)=\Xi_g(z) \cdot \tilde{z}$ and $\overline{\Psi}^{\Xi_g}(\tilde{z},z',z)=\Xi_g(z' \cdot z^{-1}) \cdot
\tilde{z}$ to prove that, for any $Y,Y'$ right-invariant vector fields,
\begin{eqnarray}
&\begin{array}{c}
Y^z(f \circ \Psi^{\Xi}_g)(\tilde{z},1_N)=Y^{z'}(f \circ \overline{\Psi}^{\Xi}_g)(\tilde{z},z,z)=(\Upsilon_g(Y))(f)(\tilde{z})
\end{array}&\label{equation_upsilon}\\
&\begin{array}{c}
Y'^z(Y^z(f \circ \Psi^{\Xi}_g))(\tilde{z},1_N)=Y'^{z'}(Y^{z'}(f \circ \overline{\Psi}^{\Xi}_g))(\tilde{z},z,z)=(\Upsilon_g(Y'))[(\Upsilon_g(Y))(f)](\tilde{z})
\end{array}&\label{equation_O}
\end{eqnarray}
where the superscript $\cdot^z,\cdot^{z'}$ means that the vector fields $Y,Y'$ apply to the $z,z'$ variables respectively.\\
The following theorem gives a complete characterization of gauge invariant L\'evy processes.

\begin{theorem}\label{theorem_characteristic3}
If a semimartingale $Z$ is a L\'evy process with characteristics $(b_0, A_0, \nu_0)$ such that its law is uniquely determined by its
characteristics, then $Z$ admits $\mathcal{G}$ as  gauge symmetry group  with action $\Xi_g$ if and only if, for any $g \in
\mathcal{G}$,
\begin{eqnarray}
b^{\alpha}_0&=&\Upsilon^{\alpha}_{g,\beta}b_0^{\beta}+\int_N{(h^{\alpha}(z')-h^{\beta}(\Xi_{g^{-1}}(z'))\Upsilon^{\alpha}_{g,\beta})\nu_0(dz')}\label{equation_characteristic5}\\
A^{\alpha\beta}_0&=&\Upsilon^{\alpha}_{g,\gamma}\Upsilon^{\beta}_{g,\delta}A^{\gamma\delta}_0\label{equation_characteristic6}\\
\nu_0&=&\Xi_{g*}(\nu_0).\label{equation_characteristic7}
\end{eqnarray}
\end{theorem}

\begin{remark}
It is important to recall that  the law of a L\'evy process
 on the Lie group $N=\mathbb{R}^n$ is always
uniquely determined by its characteristics (see \cite{Jacod2003},
Chapter II, Theorem 4.15 and following comments).
\end{remark}

In the following we provide an explicit example of use of the previous theorem, showing that our construction is a generalization of the Brownian motion case. For a more detailed discussion see \cite{Annals}.\\
Consider $N=\mathbb{R}^{n}$
and the L\'evy process with  generator given by
$$L(f)(z)=\sum_{\alpha=1}^n\frac{D}{2}\partial_{z^{\alpha}z^{\alpha}}(f)(z)+\int_{N}{(f(z+z')-f(z)-I_{|z'|<1}(z')z^{\alpha}\partial_{z^{\alpha}}(f))F(|z'|)dz'},$$
where $D \in \mathbb{R}_+$, $|\cdot |$ is the standard norm of $\mathbb{R}^n$ and $F:\mathbb{R}_+  \rightarrow \mathbb{R}_+$ is a measurable
locally bounded function such that $\int_1^{\infty}{F(r)r^{n-1}dr} < + \infty$ and $\int_0^1{F(r)r^{n+1}} < +\infty$. When $B \in SO(n)$  we have
$$\Xi_{B}(z)=B \cdot z.$$
By definition, $B$ respects the standard metric in $\mathbb{R}^n$ and so
$$\Xi_{B*}(F(|z|)dz)=\det(B)F(|B^T \cdot z|)dz=F(|z|)dz.$$
Furthermore, since $\Upsilon^{\alpha}_{\beta}=B^{\alpha}_{\beta}$ we have
\begin{eqnarray*}
&\int_N{(z^{\alpha}I_{|z|<1}(z)-\Xi^{\beta}_{B^{-1}}(z){\Upsilon}^{\alpha}_{\beta}I_{|\Xi_{B^{-1}}(z)|<1}(z))F(|z|)dz}=&\\
&=\int_N{(z^{\alpha}I_{|z|<1}(z)-(B^{-1})^{\beta}_{\gamma}B^{\alpha}_{\beta}I_{|z|<1}(z)z^{\gamma})F(|z|)dz}=0.&
\end{eqnarray*}
Hence, by Theorem \ref{theorem_characteristic3}, $Z$ admits $SO(n)$ as a gauge symmetry group with action $\Xi_B$.\\
In this case  the equation $dZ_t'=\Xi_{B_t}(dZ_t)$ is simply
$$Z'^{\alpha}_t=\int_0^t{B^{\alpha}_{\beta,s} dZ^{\beta}_s}.$$

\subsection{Gauge symmetries of discrete time independent increments processes}

In this section we focus on discrete time semimartingales with independent increments. We use here the convention adopted in Section \ref{subsection_iterated}, where we identify a discrete time process with a pure jump semimartingale having jumps at the deterministic time $t=0,1,2,...,n,...$. \\
\begin{definition}
The discrete time process $Z=(Z_0,Z_1,...,Z_n,...)$ taking values on the Lie group $N$  is a discrete time semimartingales with independent increments if $Z_n \cdot Z_{n-1}^{-1}$ is independent of $Z_0,Z_1,...,Z_n$.
\end{definition}

The law of discrete time semimartingales with independent increments is completely characterized by the measure $\nu(dz,dt)$ on $\mathbb{R}_+ \times N$ given by
$$\nu(dz,dt)=\sum_{n \in \mathbb{N}} \delta_n(dt) \mu_n(dz),$$
where $\mu_n(dz)$ is the law of the jump $\Delta Z_n=Z_n \cdot Z_{n-1}^{-1}$.\\
Therefore, a semimartingale $Z'$ has the same law of $Z$ if and only if, for any bounded continuous function on $N$,
$$f(Z'_t)-f(Z'_t)=\int_{0}^t{\int_N{f(z\cdot Z'^{-1}_{s_-})\nu(dz,dt)}}.$$
Using the previous description of the law of $Z$ we have the following theorem.
\begin{theorem}\label{theor_discrete_time}
If a semimartingale $Z$ is a discrete time semimartingale with independent increments, then $Z$ admits $\mathcal{G}$ as  gauge symmetry group  with action $\Xi_g$ if and only if, for any $g \in
\mathcal{G}$,
\begin{eqnarray*}
\nu&=&\Xi_{g*}(\nu).
\end{eqnarray*}
\end{theorem}
\begin{proof}
This is a particular case of Theorem 4.9 of \cite{Annals}.
${}\hfill$\end{proof}\\

An equivalent reformulation of the previous theorem is the following: $Z$ admits discrete time semimartingales with independent increments if and only if
$$\Xi_{g*}(\mu_n)=\mu_n$$
for any $n \in \mathbb{N}$.
\medskip
In the following we provide an interesting example, with application to iterated random maps theory, of a discrete time process admitting $O(k)$ as gauge symmetry group.\\
Consider $N=GL(k)$ and let $\Xi_g$ be the action of $O(k)$ on $N$ given by
$$\Xi_g(z)=g \cdot z \cdot g^T .$$
Let now   $Z \in GL(k)$ be discrete-time semimartingales with independent
increments characterized as follows
$$ Z_n=K_n \cdot Z_{n-1}$$
 where $K_n \in GL(k)$ are random variables independent from $Z_{1},...,Z_{n-1}$. Therefore, by Theorem \ref{theor_discrete_time}, we have that $Z$ has $O(k)$ as gauge symmetry group with action $\Xi_g$
  if and only if the distribution of $K_n \in GL(k)$ is invariant with respect to the action
 $\Xi_g$.\\
The invariance of the law of $K_n \in GL(k)$ with respect to $\Xi_g$ is exactly the invariance of the matrix random variable $K_n$ with respect
to orthogonal conjugation. This kind of random variables and
 related processes are deeply studied in random matrix theory (see, e.g., \cite{Anderson2010,Mehta2004}).\\

\section{Time symmetries}\label{section_time}

\bigskip

In this section we briefly discuss the time symmetries of a L\'evy process on a Lie group. After recalling some properties of the absolutely
continuous time change, we introduce the definition
of time symmetry of a semimartingale and we study time symmetries of L\'evy processes,  constructing some explicit examples of L\'evy processes with non-trivial
time symmetry.\\

Given  a positive adapted stochastic process  $\beta$ such that, for
any $\omega \in \Omega$, the function $\beta(\omega):t \mapsto
\beta_t(\omega)$ is absolutely continuous with strictly positive
locally bounded derivative, we define
$$\alpha_t=\inf\{s|\beta_s>t\},$$
where, as usual, $\inf(\mathbb{R}_+)=+\infty$. The process
$\alpha$ is an adapted process such that
$$\beta_{\alpha_t}=\alpha_{\beta_t}=t.$$
If $X$ is a stochastic process adapted to the filtration $\mathcal{F}_t$, we denote by $H_{\beta}(X)$ the stochastic process adapted to the
filtration $\mathcal{F}'_t=\mathcal{F}_{\alpha_t}$ such that
$$H_{\beta}(X)_t=X_{\alpha_t}.$$

Since, by assumption,
$\beta_t$ is absolutely continuous and strictly increasing, then also $\alpha_t$ is absolutely continuous and strictly increasing.  Furthermore,
 denoting by $\beta'_t$ the time derivative of $\beta_t$,  we have
$$\alpha'_t=\frac{1}{\beta'_{\alpha_t}}.$$
If $\mu$ is a random measure on $N$ adapted to the above  filtration $\mathcal{F}_t$, we can introduce a time changed random measure $H_{\beta}(\mu)$
adapted to the filtration $\mathcal{F}'_t$ such that, for any Borel set $E \subset N$,
$$H_{\beta}(\mu)([0,t] \times E)=\mu([0,\alpha_t] \times E).$$
In order to introduce a good  concept of symmetry with respect to time transformations, we have to recall some fundamental properties of absolutely
continuous random time changes with locally bounded derivative.

\begin{theorem}\label{theorem_time1}
Let $\beta_t$ be the process described above and let $Z, Z'$ be two real semimartingales, $K_t$ be a predictable process which is integrable with
respect to $Z$ and  $\mu$ be a random measure. Then
\begin{enumerate}
\item $H_{\beta}(Z)$ is a semimartingale,
\item if $Z$ is a local $\mathcal{F}_t$-martingale, then $H_{\beta}(Z)$ is a local $\mathcal{F}'_t$-martingale,
\item $H_{\beta}([Z,Z'])=[H_{\beta}(Z),H_{\beta}(Z')]$
\item $H_{\beta}(K)$ is integrable with respect to $H_{\beta}(Z)$ and $\int_0^{\alpha_t}{K_sdZ_s}=\int_0^t{H_{\beta}(K)_s dH_{\beta}(Z)_s}.$
\item if $\mu^p$ is the compensator of $\mu$, then $H_{\beta}(\mu^p)$ is the compensator of $H_{\beta}(\mu)$.
\end{enumerate}
\end{theorem}
\begin{proof}
Since the random time change $\beta$ is continuous, $\beta$ is an \emph{adapted change of time} in the meaning of \cite{Jacod1979}( Chapter X, Section b)).\\
Thank to  this remark the proofs of assertions 1, ..., 5 can be found in \cite{Jacod1979}( Chapter X, Sections b) and c)). ${}\hfill$
\end{proof}\\

Taking into account Theorem \ref{theorem_time1}, a quite natural definition of time symmetry could be the following: a semimartingale $Z$ has time symmetries if, for any $\beta$  satisfying  the
previous hypotheses, $Z$ and $H_{\beta}(Z)$ have the same law. Unfortunately, using  for example standard deterministic time changes, it is
possible to prove that the only process satisfying the previous definition is the process almost surely equal to a constant. In order to provide a different definition, admitting non-trivial examples,
we introduce a smooth action $\Gamma:\mathbb{R}_+ \times N \rightarrow N$ of the group $\mathbb{R}_+$ on $N$. We write $\gamma_r:\mathfrak{n} \rightarrow \mathfrak{n}$ for the linear action of $\mathbb{R}_+$ on $\mathfrak{n}$ such that
\begin{eqnarray*}
\gamma_r(Y)&=&\partial_a(\Gamma_r \circ \Phi_a)(1_N)|_{a=0},
\end{eqnarray*}
where $\Phi_a$ is the flow of a right invariant vector field $Y$ on $N$.
If we assume that
$$ \partial_b(\partial_a(\Gamma_r \circ \Phi_a \circ \Phi'_b))(1_N)|_{a=0,b=0}=0, $$
we obtain
\begin{eqnarray}
&\begin{array}{c}
Y^z(f \circ \Psi^{\Gamma}_r)(\tilde{z},1_N)=Y^{z'}(f \circ \overline{\Psi}^{\Gamma}_r)(\tilde{z},z,z)=(\gamma_r(Y))(f)(\tilde{z})
\end{array}&\label{equation_upsilon1}\\
&\begin{array}{c}
Y'^z(Y^z(f \circ \Psi^{\Gamma}_r))(\tilde{z},1_N)=Y'^{z'}(Y^{z'}(f \circ \overline{\Psi}^{\Gamma}_r))(\tilde{z},z,z)=(\gamma_r(Y'))[(\gamma_r(Y))(f)](\tilde{z}).
\end{array}&\label{equation_O1}
\end{eqnarray}

\begin{definition}\label{definition_time}
Let $Z$ be a semimartingale on a Lie group $N$ and let $\Gamma_\cdot:N \times \mathbb{R}_+ \rightarrow N$ be an $\mathbb{R}_+$ action such that
$\Gamma_r(1_N)=1_N$ for any $r \in \mathbb{R}_+$. We say that $Z$ has a time symmetry with action $\Gamma_r$ with respect to the filtration $\mathcal{F}_t$ if
$$dZ'_t=H_{\beta}(\Gamma_{\beta'_t}(dZ_t))$$
has the same law of $Z$ for any $\beta_t$  satisfying the previous hypotheses and such that $\beta'_t$ is a $\mathcal{F}_t$-predictable locally
bounded process in $\mathbb{R}_+$.
\end{definition}

\begin{remark}
The request that $\beta'_t$ is a locally bounded process in $\mathbb{R}_+$ ensures  that $\beta'_t(\omega) \geq c(\omega) >0$ for some $c(\omega)
\in \mathbb{R}_+$ and for $t$ in compact sets.
\end{remark}

\begin{lemma}\label{lemma_time2}
If $(X,Z)$ is a solution to the SDE $\Psi_{K_t}$ and  $\beta$ is an absolutely continuous process such that $\beta'_t$ is locally bounded in
$\mathbb{R}_+$, then $(H_{\beta}(X),H_{\beta}(Z))$ is a solution to the SDE $\Psi_{H_{\beta}(K)_t}$.
\end{lemma}
\begin{proof}
The thesis is a simple consequence of Definition \ref{definition_solution} and Theorem \ref{theorem_time1}, point 4. ${}\hfill$ \end{proof}\\

We now restrict our attention to L\'evy processes on $N$, recalling  some general results about L\'evy processes with  time symmetries and providing
explicit  examples.

\begin{theorem}\label{theorem_time3}
If $Z$ is a L\'evy process with characteristics $(b_0 ,A_0 , \nu_0)$, then $Z$ admits a time symmetry with action $\Gamma_r$ if and only
if,  for any fixed  $r \in \mathbb{R}_+$,
\begin{eqnarray}
b_0^{\alpha}&=&\frac{1}{r}\left(\gamma^{\alpha}_{r,\beta}b^{\beta}_0\right)+
\frac{1}{r}\int_N{(h^{\alpha}(z')-h^{\beta}(\Gamma_{r^{-1}}(z))\gamma^{\alpha}_{r,\beta})\nu_0(dz')}\label{equation_timec1}\\
A_0^{\alpha\beta}&=&\frac{1}{r} \gamma^{\alpha}_{r,\gamma}\gamma^{\beta}_{r,\delta}A^{\gamma\delta}_0\label{equation_timec2}\\
\nu_0(dz)&=&\frac{1}{r}\Gamma_{r*}(\nu_0(dz)).\label{equation_timec3}
\end{eqnarray}
\end{theorem}
\begin{proof}
First we note that, putting $d\tilde{Z}_t=H_{\beta}(\Gamma_{\beta'}(dZ_t))$, using Theorem \ref{theorem_time1} and the definition of stochastic characteristics (see \cite{Annals,Jacod2003}), the stochastic characteristic of $\tilde{Z}$ are given by
\begin{eqnarray}
\tilde{b}_t^{\alpha}&=&H_{\beta}\left(\int_0^t \left(\gamma^{\alpha}_{\beta'_s,\beta}b^{\beta}_0+
\int_N{(h^{\alpha}(z')-h^{\beta}(\Gamma_{\beta'^{-1}_s}(z))\gamma^{\alpha}_{\beta'_s,\beta})\nu_0(dz')}\right)ds\right)\label{equation_timec4}\\
\tilde{A}_t^{\alpha\beta}&=&H_{\beta}\left(\int_0^t{\gamma^{\alpha}_{\beta'_s,\gamma}\gamma^{\beta}_{\beta'_s,\delta}A^{\gamma\delta}_0ds}\right)\label{equation_timec5}\\
\tilde{\nu}(dt,dz)&=&H_{\beta}(\Gamma_{\beta'_t}(\nu_0(dz))dt)\label{equation_timec6}
\end{eqnarray}
(see also \cite{Albeverio2017} for a complete proof of this fact).\\
\noindent If $\beta'_t=r \in \mathbb{R}_+$, and so $\beta_t=r t$,  using the fact that $\alpha_t=\frac{t}{r}$ we obtain
\begin{eqnarray*}
\tilde{b}_t^{\alpha}&=&\int_0^t \left(\gamma^{\alpha}_{r,\beta}b^{\beta}_0+
\int_N{(h^{\alpha}(z')-h^{\beta}(\Gamma_{r^{-1}}(z))\gamma^{\alpha}_{r,\beta})\nu_0(dz')}\right)\frac{ds}{r}\\
\tilde{A}_t^{\alpha\beta}&=&\int_0^t{\gamma^{\alpha}_{r,\gamma}\gamma^{\beta}_{r,\delta}A^{\gamma\delta}_0\frac{ds}{r}}\\
\tilde{\nu}(dt,dz)&=&\frac{1}{r}\Gamma_{r,*}(\nu_0(dz))dt.
\end{eqnarray*}
Using the fact that, since the sigma algebras generated by $Z_t$ and $\tilde{Z}$ coincide,  $\tilde{Z}$ has the same law of $Z$ only if $\tilde{b}_t=b_0t, \tilde{A}_t=A_0t, \tilde{\nu}(dt,dz)=\nu_0(dz)dt$, we get equations \refeqn{equation_timec1}, \refeqn{equation_timec2} and \refeqn{equation_timec3}.\\
Conversely suppose that equations \refeqn{equation_timec1}, \refeqn{equation_timec2} and \refeqn{equation_timec3} hold
 and let $\beta'_t$ be an elementary process. Then, by equations \refeqn{equation_timec4}, \refeqn{equation_timec5} and \refeqn{equation_timec6}, we have that the stochastic characteristics of $\tilde{Z}$ are exactly $b_0t,A_0t,\nu_0(dz)dt$. This means that $\tilde{Z}$ is a L\'evy process with the same law of $Z$. Since the elementary processes are dense in the set of predictable processes, we have the thesis.
 ${}\hfill$ \end{proof}\\

We provide now an example of time symmetric L\'evy processes. When $N=\mathbb{R}^n$ (homogeneous),  $\alpha$-stable processes are well known since their generator is the
fractional
Laplacian, and  they can be obtained by a subordination from a Brownian motion (see, e.g., \cite{Albeverio2007,Applebaum2004}). \\
The homogeneous $\alpha$-stable processes are L\'evy processes in $\mathbb{R}^n$  depending on a parameter $\alpha \in (0,2)$. 
The process $Z$ is a pure jump L\'evy process with L\'evy measure
$$\nu_{\alpha}(dz)=\frac{1}{|z|^{n+\frac{\alpha}{2}}}dz,$$
where $|\cdot|$ is the standard norm of $\mathbb{R}^n$, $dz$ is the Lebesgue measure and $A_0=0,b_0=0$.\\
The generator $L_{\alpha}$ of an $\alpha$-stable process is
$$L_{\alpha}(f)(z)=\int_{\mathbb{R}^n}{\left(f(z+z')-f(z)-I_{|z'|<1}(z')\left(
z'^{\beta}\partial_{z^{\beta}}(f)(z)\right)\right)\nu_{\alpha}(dz')}.$$
If we consider the following action
$$\Gamma^{\alpha}_r(z)=r^{\frac{1}{\alpha}}z,$$
we have that
$$\Gamma_{r*}(\nu_0)=\frac{1}{2}\nu_0. $$
Furthermore, using the fact that $\nu_0$ is invariant with respect to rotations, we have
$$\int_{\mathbb{R}^n}(I_{B}(z')-I_{\Gamma_{1/r}(B)}z'^{\alpha})\nu(dz')=0.$$
Thus if we exploit Theorem \ref{theorem_time3}, we obtain that the $\alpha$-stable processes are time symmetric with respect to the action $\Gamma_r$.

\section{Symmetry and invariance properties of a SDE with jumps}\label{section_symmetry}

\subsection{Stochastic transformations}\label{Subsec1}

Let $\mathcal{C}(\mathbb{P}_0)$ (or simply $\mathcal{C}$) be the class of c\acc{a}dl\acc{a}g semimartingales $Z$ on a Lie group $N$  inducing
the same probability measure on $\mathcal{D}([0,T],N)$ (the metric space of c\acc{a}dl\acc{a}g functions taking values in $N$). In order to
generalize to the semimartingale case the notion of weak solution to a SDE driven by a Brownian motion, we introduce the following

\begin{definition}\label{definition_transformation1}
 Given  a semimartingale $X$ on $M$ and  a semimartingale $Z$ on $N$ such that
$Z \in \mathcal{C}$, the pair $(X,Z)$  is called a \emph{process of class $\mathcal{C}$} on $M$. \\
A process $(X,Z)$ of class $\mathcal{C}$ which is a solution to the geometrical SDE $\Psi$ is called a \emph{solution of class $\mathcal{C}$ } to
$\Psi$.
\end{definition}

We remark that if $(X,Z)$ and $(X',Z')$ are two solutions of
 class $\mathcal{C}$ and if $X_0$ and  $X'_0$ have the same law, then also  $X$ and $X'$ have the same law.\\

In this section we  define a set of transformations which transform a process of class $\mathcal{C}$ into a new process of class $\mathcal{C}$.
This set of transformations depends on the properties of the processes
belonging to the class $\mathcal{C}$.\\
We start by describing the case of processes in $\mathcal{C}$ admitting a gauge-symmetry group $\mathcal{G}$ with action
$\Xi_g$ and a time symmetry with action $\Gamma_r$. Afterwords,  we discuss how to extend our approach to more general situations.\\

\begin{definition}
A stochastic transformation from a manifold $M$ into a manifold $M'$ is a triple $(\Phi,B,\eta)$, where $\Phi$ is a diffeomorphism of $M$ into $M'$, $B:M \rightarrow
\mathcal{G}$ is a smooth function and $\eta:M \rightarrow \mathbb{R}_+$ is a positive smooth function. We denote by $S(M,M')$ the set of stochastic
transformations of $M$ into $M'$.
\end{definition}

A stochastic transformation defines a map between the set of stochastic processes of class $\mathcal{C}$ on $M$ into the set of stochastic
processes of class $\mathcal{C}$ on $M'$. The action of the stochastic transformation $T \in S(M,M')$ on the stochastic process
$(X,Z)$ is denoted by $(X',Z')=P_T(X,Z)$, and is defined as follows:
\begin{eqnarray*}
X'&=&\Phi\left[H_{\beta^{\eta}}(X)\right]\\
dZ_t'&=&H_{\beta^{\eta}}\left\{\Xi_{B(X_t)}\left[\Gamma_{\eta(X_t)}(dZ_t)\right]\right\},
\end{eqnarray*}
 where $\beta^{\eta}$ is the random time change given by
$$\beta^{\eta}_t=\int_0^t{\eta(X_s)ds}.$$
The second step is to define  an action of a stochastic transformation on the set of geometrical SDEs. This action transforms a geometrical SDE
$\Psi$ on $M$ into the geometrical SDE $\Psi'=E_T(\Psi)$ on $M'$ defined by
$$\Psi'(x,z)=\Phi\left\{\Psi\left[\Phi^{-1}(x),(\Gamma_{(\eta(\Phi^{-1}(x))^{-1}} \circ \Xi_{(B(\Phi^{-1}(x)))^{-1}})(z)\right]\right\}.$$

\begin{theorem}\label{theorem_symmetry3}
If $T \in S(M,M')$ is a stochastic transformation and  $(X,Z)$ is a  class $\mathcal{C}$ solution  to the geometrical SDE $\Psi$,
then $P_T(X,Z)$ is a class $\mathcal{C}$ solution  to the geometrical SDE $E_T(\Psi)$.
\end{theorem}
\begin{proof}
The fact that $P_T(X,Z)$ is a process of class $\mathcal{C}$ follows from the symmetries  of $Z$, which are the gauge
symmetry group $\mathcal{G}$ with action $\Xi_g$ and the time symmetry with action $\Gamma_r$.\\
The fact that, if $(X,Z)$ is a solution to $\Psi$, then $P_T(X,Z)$ is a solution to $E_T(\Psi)$, follows from Theorem
\ref{theorem_geometrical1}, Theorem \ref{theorem_gauge1} and Lemma \ref{lemma_time2}.
${}\hfill$\end{proof}

\begin{remark}\label{remark_senza}
If $\mathcal{C}$ contains semimartingales admitting a gauge symmetry group $\mathcal{G}$ but without time symmetry,  the stochastic
transformation reduces to a pair $(\Phi,B)$ and the action on processes and SDEs is the same as in the general case with $\Gamma_r=Id_N$. The
same argument can be applied in the case of $\mathcal{C}$ containing semimartingales with only time symmetry.
\end{remark}

In the case of semimartingales
without neither  gauge  nor  time symmetries,  the stochastic transformations can be identified with the diffeomorphisms $\Phi:M \rightarrow M'$
and the action on the processes is $P_T(X,Z)=(\Phi(X),Z)$. Since these kinds of transformations do not change the driving process $Z$ and  play
a special role in the theory of symmetries we call  a stochastic transformation of the form
$(\Phi,1_N,1)$ a \emph{strong stochastic transformation}. Hereafter, in order to stress the difference between strong and more general  stochastic transformations, when necessary we use the name \emph{weak stochastic transformations} for the latter.

\subsection{The geometry of stochastic transformations}

In this subsection we  prove that stochastic transformations have some interesting geometric properties, allowing us to extend  to c\acc{a}dl\acc{a}g-semimartingales-driven SDEs the results given in \cite{DMU1} for SDEs driven by  Brownian motions. \\
In order to keep holding some crucial  geometric properties, in the following we require   an additional property on the maps $\Xi_g$ and
$\Gamma_r$, i.e. the commutation of the   two group actions $\Xi_g$ and $\Gamma_r$. In particular we suppose that
\begin{equation}\label{ST1}
\Xi_g(\Gamma_r(z))=\Gamma_r(\Xi_g(z)),
\end{equation}
for any $z \in N$, $g \in \mathcal{G}$ and $r \in \mathbb{R}_+$.\\

We can define a composition between two stochastic transformations $T \in S(M,M')$ and $T' \in S(M',M'')$, where
$T=(\Phi,B,\eta)$ and $T'=(\Phi',B',\eta')$,  as
\begin{equation}\label{equation_symmetry1}
T' \circ T=(\Phi' \circ \Phi,(B' \circ \Phi) \cdot B, (\eta' \circ \Phi)\eta ).
\end{equation}
The above composition has a nice geometrical interpretation. If we denote by $\mathcal{H}=\mathcal{G} \times \mathbb{R}_+$, a stochastic transformation from $M$ into $M'$ can be identified with an isomorphism from the
trivial right principal bundle $M \times \mathcal{H}$ into the trivial right principal bundle $M' \times \mathcal{H}$ which preserves the
principal bundle structure. Exploiting this identification and the natural isomorphisms composition, we obtain  formula
\refeqn{equation_symmetry1}
(see \cite{DMU1} for the case $\mathcal{G}=SO(m)$).\\
Given a stochastic transformation $T \in S(M,M')$, composition \refeqn{equation_symmetry1}  permits to define an inverse $T^{-1} \in S(M',M)$
as follows
$$T^{-1}= (\Phi^{-1},(B \circ \Phi^{-1})^{-1},(\eta \circ \Phi^{-1})^{-1}).$$
Hence the set $S(M):=S(M,M)$ is a group with respect to the composition $\circ$ and the identification
 of $S(M)$ with $\Iso(M \times \mathcal{H}, M \times \mathcal{H})$ (which is a closed subgroup of the group of diffeomorphisms of $M \times \mathcal{H}$) suggests to consider  the corresponding Lie algebra $\mathcal{V}(M)$.\\
Given a one parameter group  $T_a=(\Phi_a,B_a,\eta_a) \in S(M)$,  there exist a vector field $Y$ on $M$, a smooth function $C:M
\rightarrow \mathfrak{g}$ (where $\mathfrak{g}$ is the Lie algebra of $\mathcal{G}$), and a smooth function $\tau:M \rightarrow \mathbb{R}$ such
that
\begin{equation}\label{equation_infinitesimal_SDE1}\begin{array}{ccc}
Y(x)&:=&\partial_a(\Phi_a(x))|_{a=0}\\
C(x)&:=&\partial_a(B_a(x))|_{a=0}\\
\tau(x)&:=&\partial_a(\eta_a(x))|_{a=0}.
\end{array}
\end{equation}
So if  $Y,C,\tau$ are  as above, the one parameter solution $(\Phi_a,B_a,\eta_a)$ to the equations
\begin{equation}\label{equation_infinitesimal_SDE2}\begin{array}{rcl}
\partial_a(\Phi_a(x))&=&Y(\Phi_a(x))\\
\partial_a(B_a(x))&=&R_{B_a(x)*}(C(\Phi_a(x)))\\
\partial_a(\eta_a(x))&=&\tau(\Phi_a(x))\eta_a(x),
\end{array}
\end{equation}
with initial condition $\Phi_0=id_M$, $B_0=1_{\mathcal{G}}$ and $\eta_0=1$, is a one parameter group in $S(M)$. For this reason we identify the elements of $\mathcal{V}(M)$ with the triples $(Y,C,\tau)$.

\begin{definition}\label{def_infstochtransf}
A triple $V=(Y,C,\tau)\in \mathcal{V}(M)$, where  $Y$ is a vector field on $M$,  $C:M \rightarrow \mathfrak{g}$  and  $\tau:M \rightarrow
\mathbb{R}$  are smooth functions, is an  infinitesimal stochastic transformation. If $V$ is of the form  $V=(Y,0,0)$  we call $V$ a strong
infinitesimal stochastic transformation, as the corresponding one-parameter group is a group of strong stochastic transformations.
\end{definition}

As we mentioned at the end of subsection \ref{Subsec1},  in order to stress the difference between strong and more general infinitesimal stochastic transformations, when necessary we use the name \emph{weak infinitesimal stochastic transformations}.\\
Since $\mathcal{V}(M)$ is a Lie subalgebra of the set of vector fields on $M \times \mathcal{H}$, the standard Lie brackets between vector
fields on $M \times \mathcal{H}$ induces some Lie brackets on $\mathcal{V}(M)$. Indeed, if $V_1=(Y_1,C_1,\tau_1),V_2=(Y_2,C_2,\tau_2) \in
\mathcal{V}(M)$ are two infinitesimal stochastic transformations, we have
\begin{equation}\label{equation_infinitesimal_SDE5}
\left[V_1,V_2\right]=(\left[Y_1,Y_2\right],Y_1(C_2)-Y_2(C_2)-\{C_1,C_2\},Y_1(\tau_2)-Y_2(\tau_1)),
\end{equation}
where $\{\cdot,\cdot\}$ denotes the usual commutator between elements of $\mathfrak{g}$.\\
Furthermore the identification of  $T=(\Phi,B,\eta) \in S(M,M')$  with  $F_T\in \Iso(M \times \mathcal{H}, M' \times \mathcal{H})$
allows us to define  the push-forward $T_*(V)$  of $V \in \mathcal{V}(M)$ as
\begin{equation}\label{equation_infinitesimal_SDE4}
(\Phi_*(Y) ,(Ad_{B}(C)+R_{B^{-1}*}(Y(B))) \circ \Phi^{-1},(\tau+Y(\eta)\eta^{-1})\circ \Phi^{-1}),
\end{equation}
where $Ad$ denotes the adjoint operation and the symbol $Y(B)$  the push-forward of $Y$ with respect to the map $B:M \rightarrow \mathcal{G}$.\\
Analogously, given $V' \in \mathcal{V}(M')$, we can consider  the pull-back of $V'$  defined as $T^*(V')=(T^{-1})_*(V')$.
Any Lie algebra of general infinitesimal stochastic transformations satisfying a non-degeneracy condition, can be locally
transformed, by action of the push-forward of a suitable stochastic transformation $T \in S(M)$,  into a Lie algebra of strong
infinitesimal stochastic transformations (see Theorem \ref{theorem_infinitesimal_SDE1} below).

\subsection{Symmetries of a SDE with jumps}

\begin{definition}\label{definition_symmetries}
A stochastic transformation $T \in S(M)$ is a symmetry of the SDE $\Psi$ if,
for any process $(X,Z)$ of class $\mathcal{C}$ solution to the SDE $\Psi$,  also $P_T(X,Z)$ is a solution to the SDE $\Psi$.\\
An infinitesimal stochastic transformation $V \in \mathcal{V}(M)$ is a symmetry of the SDE $\Psi$ if the one-parameter group of
stochastic transformations $T_a$ generated by $V$ is a group of symmetry of the SDE $\Psi$.
\end{definition}

In the following we often use the name \emph{strong symmetries} and \emph{weak symmetries} for stressing the fact that the considered symmetry is a strong or a weak (infinitesimal or finite) stochastic transformation in the sense mentioned above, after Definition \ref{def_infstochtransf}.

\begin{remark}
We can give also a local version of Definition \ref{definition_symmetries}: a stochastic transformation $T \in S(U,U')$, where
$(U,U')$ are two open sets of $M$, is a symmetry of $\Psi$ if $P_T$ transforms solutions to $\Psi|_{U}$ into solutions to
$\Psi|_{U'}$.\\
In this case it is necessary to stop the solution process $X$ and the driving semimartingale $Z$ with respect to a suitably adapted  stopping
time.
\end{remark}

\begin{theorem}\label{theorem_symmetry1}
A sufficient condition for a stochastic transformation $T\in S(M)$ to be a symmetry of the SDE $\Psi$ is that $E_T(\Psi)=\Psi$.
\end{theorem}
\begin{proof}
This is an easy application of Theorem \ref{theorem_symmetry3}.
${}\hfill$\end{proof}\\

A natural question arising from  previous discussion is whether the sufficient condition of Theorem \ref{theorem_symmetry1} is also necessary. Unfortunately,
 even for Brownian motion driven SDEs there are counterexamples
(see \cite{DMU1}, where the determining equations for symmetries of SDE are different from the equations found here). The reason for this fact is
that, for a general law of the driving semimartingale in the class $\mathcal{C}$, it is possible to find two different geometrical SDEs $\Psi $ and
 $\Psi'$ with the same solution $(X,Z)$ of class $\mathcal{C}$, i.e.  any solution $(X,Z)$ of $\Psi$
is also a solution of $\Psi'$ and viceversa. \\
Exploiting this fact it is possible to find suitable conditions in order to prove  the converse of Theorem \ref{theorem_symmetry1}.\\
In  the following we say that a semimartingale $Z$ in the class $\mathcal{C}$ and with (stochastic) characteristic triplet $(b,A,\nu)$ (see \cite{Annals}) has \emph{jumps of
any size} if the support of $\nu$ is all of $N \times \mathbb{R}_+$ with positive probability. If we restrict to L\'evy processes this is equivalent to require that the support of $\nu_0$ coincides with $N$. In the case where $Z$ is a L\'evy process with characteristic $(b_0,A_0,\nu_0)$ the previous request is equivalent to require that the support of the measure $\nu_0$ is all of $N$.

\begin{lemma}\label{lemma_full}
Given a semimartingale $Z$ in the class $\mathcal{C}$ with jumps of any size and such that the stopping time $\tau$ of the first jump
 is almost surely strictly positive,  if $(X,Z)$ is a solution to both the SDEs $\Psi$ and $\Psi'$ such that $X_0=x_0 \in M$ almost surely,  then
$\Psi(x_0,z)=\Psi'(x_0,z)$ for any $z \in N$.
\end{lemma}
\begin{proof}
Consider the semimartingale $S^f_t=f(X_t)$, where $f \in \cinf(M)$ is a bounded smooth function. Given  a bounded smooth function $h \in
\cinf(\mathbb{R})$ such that $h(x)=0$ for $x $ in a neighborhood of $0$, we define the (special) semimartingale
$$H^{h,f}_t=\sum_{0 \leq s \leq t}h(\Delta S^f_s).$$
Since the jumps $\Delta S^f_t$ of $S^f$ are exactly $\Delta S^f_t=f(\Psi(X_{t_-},\Delta Z_t))-f(X_{t_-})$ or, equivalently, $\Delta
S^f_t=f(\Psi'(X_{t_-},\Delta Z_t))-f(X_{t_-})$ we have that \small
$$H^{h,f}_t=\int_{N \times [0,t]}{h(f(\Psi(X_{s_-},
z))-f(X_{s_-}))\mu^Z(ds,dz)}=\int_{N \times [0,t]}{h(f(\Psi'(X_{s_-}, z))-f(X_{s_-}))\mu^Z(ds,dz)}.$$ \normalsize
Since $H^{h,f}$ is a special
semimartingale there exists a unique (up to $\mathbb{P}$ null sets) predictable process $R^{h,f}$ of bounded variation such that
$H^{h,f}_t-R^{h,f}_t$ is a local martingale. By the definition of characteristic measure $\nu$  it is simple to prove that
\small
$$R^{h,f}_t=\int_{N \times [0,t]}{h(f(\Psi(X_{s_-}, z))-f(X_{s_-}))\nu(ds,dz)}=\int_{N \times [0,t]}{h(f(\Psi'(X_{s_-}, z))-f(X_{s_-}))\nu(ds,dz)}.$$
\normalsize This means that
$$\int_{N \times [0,t]}{(h(f(\Psi(X_{s_-}, z))-f(X_{s_-}))-h(f(\Psi'(X_{s_-}, z))-f(X_{s_-})))\nu(ds,dz)}$$
is a semimartingale almost surely equal to $0$. Since $X_{t_-}$ is a continuous function for $t\leq \tau$ and  the support of $\nu$ is all $N
\times \mathbb{R}_+$, in a set of positive measure, there exists a set of positive probability such that $h(f(\Psi(X_{t_-},
z))-f(X_{t_-}))-h(f(\Psi'(X_{t_-}, z))-f(X_{t_-}))=0$ for any $z \in N$. Taking the limit $t \rightarrow 0$ we obtain $h(f(\Psi(x_0,
z))-f(x_0))=h(f(\Psi'(x_0, z))-f(x_{0}))$. Since $h,f$ are generic functions, $\Psi(x_0,z)=\Psi'(x_0,z)$ for any $z \in N$.
${}\hfill$\end{proof}

\begin{theorem}\label{theorem_full}
Given a manifold $M$, under the hypothesis of Lemma \ref{lemma_full}, a stochastic transformation $T \in S(M)$  is a symmetry of a SDE $\Psi$ if and only if
$E_T(\Psi)=\Psi$.
\end{theorem}
\begin{proof}
The if part is exactly Theorem \ref{theorem_symmetry1}. \\
Conversely, suppose  that $T$ is a symmetry of $\Psi$ and put $\Psi'=E_T(\Psi)$. If $X^{x_0}$ denotes  the unique solution to the SDE $\Psi$
driven by the semimartingale $Z$ such that $X^{x_0}=x_0$ almost surely, put $(X',Z')=E_T(X^{x_0},Z)$.  By definition of symmetry $(X',Z')$ is a
solution to $\Psi$ and, by Theorem \ref{theorem_symmetry3}, it is a solution to $\Psi'$. Since $X'_0=\Phi(x_0)$ almost surely, using Lemma
\ref{lemma_full} we obtain that $\Psi(\Phi(x_0),z)=\Psi'(\Phi(x_0),z)$. Since $\Phi$ is a diffeomorphism and $x_0 \in M$ is a generic point this
concludes the proof.
${}\hfill$\end{proof}

\begin{remark}
We propose here two possible generalizations of Theorem \ref{theorem_full}\\
First we can suppose that $Z$ is a purely discontinuous semimartingale and that $b^{\alpha}_t=A^{\alpha,\beta}_t=0, \forall t \geq 0$ with truncated
functions $h^{\alpha}=0$. In this case, if  the support of $\nu$ is $J \times \mathbb{R}_+$ almost surely,  the stochastic transformation  $T$
is a symmetry of the SDE $\Psi$ if and only if $E_T(\Psi)(x,z)=\Psi(x,z)$ for any $z \in J$. The proof of the necessity of the condition is
equal to the one in Lemma \ref{lemma_full} and Theorem \ref{theorem_full}, instead the proof of the sufficiency part is essentially based on the
fact that $Z$ is a pure jump process.
This case includes, for example, the Poisson process.\\
The second generalization covers the important case of continuous semimartingales. An example of the theorem which could be obtained in this
case is Theorem 17 in \cite{DMU1} that, in our language,   can be reformulated as follows: $T$ is a symmetry of $\Psi$ driven by a Brownian
motion $Z^2,...,Z^m$ and by the time $Z^1_t=t$ if and only if $\partial_{z^{\alpha}}(\Psi)(x,0)=\partial_{z^{\alpha}}(E_T(\Psi))(x,0)$ for
$\alpha=2,...,m$ and $\partial_{z^1}(\Psi)(x,0)+\frac{1}{2}\sum_{\alpha=2}^m
\partial_{z^{\alpha}z^{\alpha}}(\Psi)(x,0)=\partial_{z^1}(E_T(\Psi))(x,0)
+\frac{1}{2}\sum_{\alpha=2}^m \partial_{z^{\alpha}z^{\alpha}}(E_T(\Psi))(x,0)$. \\
\end{remark}

In order to provide an explicit formulation of the \emph{determining equations} for the infinitesimal symmetries of a SDE $\Psi$, we prove the
following proposition.

\begin{proposition}\label{proposition_determining}
A sufficient condition for an infinitesimal stochastic transformation  $V\in \mathcal{V}(M)$, generating a one-parameter group $T_a$ of stochastic transformations,
to be an infinitesimal symmetry of a SDE $\Psi$ is that
\begin{equation}\label{equation_determining1}
\partial_a(E_{T_a}(\Psi))|_{a=0}=0.
\end{equation}
When the hypotheses of Theorem \ref{theorem_full} hold, condition \refeqn{equation_determining1} is also necessary.
\end{proposition}
\begin{proof}
We prove that if equation \refeqn{equation_determining1} holds, then $E_{T_a}(\Psi)=\Psi$ for any $a \in \mathbb{R}$. Defining
$\Psi(a,x,z)=E_{T_a}(\Psi)$, the function $\Psi(a,x,z)$ solves a partial differential equation of the form
\begin{equation}\label{equation_determining3}
\partial_a(\Psi(a,x,z))=L(\Psi(a,x,z))+F(\Psi(a,x,z),x,z), \quad x\in M, z\in N
\end{equation}
where $L$ is a linear first order scalar differential operator in $\partial_x,\partial_z$ and $F$ is a smooth function. It is possible to prove,
exploiting standard techniques of characteristics of first order PDE (see \cite{DMU1,DMU2}), that equation \refeqn{equation_determining3}
admits a unique local solution as evolution PDE in the time parameter $a$ for any smooth initial value $\Psi(0,x,z)$.\\
Since   $\Psi(0,x,z)=E_{T_0}(\Psi)(x,z)=\Psi(x,z)$ and
$L(\Psi(x,z))+F(\Psi(x,z),x,z)=\partial_a(E_{T_a}(\Psi))|_{a=0}=0$, we have that $E_{T_a}(\Psi)(x,z)=\Psi(a,x,z)=\Psi(x,z)$. \\
The necessity of condition \refeqn{equation_determining1} under the hypotheses of Theorem \ref{theorem_full} is trivial since, by Theorem
\ref{theorem_full}, we must have $E_{T_a}(\Psi)=\Psi$.
${}\hfill$\end{proof}\\

Given a coordinate system $x^i$ on $M$ and  $z^{\alpha}$ on $N$, we can use Proposition \ref{proposition_determining} to rewrite equations \refeqn{equation_determining1} in a more explicit form.
If we denote by $K_1,...,K_r$ the vector fields  on $N$ generating the action $\Xi_g$ of
$\mathcal{G}$ on $N$ and by $H$ the vector field generating  the action $\Gamma_r$ of $\mathbb{R}_+$ on $N$,  with any
infinitesimal stochastic transformation $V=(Y,C,\tau)$ we can associate a vector field $Y$ on $M$, a function $\tau$ and $r$ functions
$C^1(x),...,C^r(x)$
 which correspond to the components of $C$ with respect to the basis $K_1,...,K_r$. Therefore,
 the vector fields $Y$ and  $K_1,...,K_r,H$ are of the form
$$
Y=Y^i(x)\partial_{x^i} \ \  K_{\ell}=K_{\ell}^{\alpha}(z)\partial_{z^{\alpha}} \ \ H=H^{\alpha}(z)\partial_{z^{\alpha}}
$$
and we obtain the following theorem.

\begin{theorem}[Determining equations]
A sufficient condition such that the infinitesimal stochastic transformation $(Y,C,\tau)$ is a symmetry of the SDE $\Psi$ is that
\begin{equation}\label{equation_determining2}
Y^i(\Psi(x,z))-Y^j(x)\partial_{x^j}(\Psi^i)(x,z)-\tau(x) H^{\alpha}(z)\partial_{z^{\alpha}}(\Psi^i)(x,z)-C^{\ell}(x) K^{\alpha}_{\ell}(z)
\partial_{z^{\alpha}}(\Psi^i)(x,z)=0, 
\end{equation}
where $x\in M, z\in N$ and  $\Psi^i(x,z)=x^i \circ \Psi$ and $i=1,...,m$.
Furthermore the previous condition is also necessary if $Z$ has jumps of
any size.
\end{theorem}
\begin{proof}
The proof of the necessary part of the theorem is obtained by writing the condition of Proposition \ref{proposition_determining} in coordinates. The sufficient part is a consequence of Theorem \ref{theorem_full}.\hfill\end{proof}

\medskip

In the literature of symmetries of deterministic differential equations,  equations
\refeqn{equation_determining2} are usually called \emph{determining equations}(see, e.g., \cite{Olver1993,Stephani1989}). It is important to
note some differences with respect to the determining equations of ODEs or also of Brownian-motion-driven SDEs (see \cite{DMU1}). Indeed, in the
deterministic case and in the Brownian  motion case the determining equations are linear and local overdetermined first order differential
equations both in the infinitesimal transformation coefficients and in the equation coefficients. Instead equations
\refeqn{equation_determining2} are \emph{linear non-local} differential equations in the coefficients $Y^i,\tau,C^{\ell}$ of the infinitesimal
transformation $V$, and they are \emph{non-linear local} differential equations in the coefficient
$\Psi^i$ of the SDE. \\

\subsection{Reduction and reconstruction through infinitesimal symmetries}

In this section we propose a theorem of reduction and reconstruction for symmetric SDEs. First of all we introduce the notion of \emph{triangular SDEs.}

\begin{definition}
We say that a geometrical SDE $\Psi$ is triangular with respect to the first $r$ coordinates $x^1,...,x^r$ if $\Psi$  is of the form
\begin{eqnarray*}
\Psi^i(x,z)=x^i+B^i(x^{i+1},...,x^m,z) \ \ \text{for }i=1,...,r\\
\Psi^i(x,z)=B^i(x^{r+1},...,x^m,z) \ \ \text{for }i=r+1,...,m,
\end{eqnarray*}
for some maps $B^i$.
\end{definition}

The name \virgolette{triangular SDE} follows from the fact that if we write explicitly the differential relations satisfied by the processes $X^i_t$ they are triangular with respect to the process $X^1_t,...,X^r_t$. \\
In particular this implies that the processes $X^1,...,X^r$ \emph{can be reconstruct} from the reduced processes $X^{r+1},...,X^{m}$ using only iterated It\^o integrals. In the following we prove that if a geometrical SDE $\Psi$ admits a non-degenerate $r$ dimensional solvable Lie algebra of infinitesimal symmetries there exits a (generally local) stochastic transformation $T=(\Phi,B,\eta)$ such that $E_T(\Psi)$ is in triangular form. In other words we can reduce the $m$ dimensional  solution $X$ to the equation $\Psi$ to a lower dimensional process $X'^{r+1},...,X'^{m}$ satisfying an $m-r$ dimensional equation and we can reconstruct the process $X$ by using only iterated It\^o integrals and by inverting the stochastic transformation $T$. A consequence of this theorem is that, if $r=m$, the process $X$ can be reconstructed from a deterministic process, in other words it can be reconstructed by \emph{quadratures}. These theorem can be seen as a generalization of the analogous results for symmetric deterministic ODEs (see \cite{Olver1993,Stephani1989}). \\
In order to establish these results we need some preliminary theorems and definition.

\begin{definition}\label{definition_regular_vectors}
A set of vector fields  $Y_1,...,Y_k$ on $M$ is  \emph{regular} on $M$ if, for any $x \in M$, the vectors $Y_1(x),...,Y_k(x)$ are linearly independent.
\end{definition}

\begin{definition}\label{definition_solvable_coordinate}
Let $Y_1,...,Y_r$ be a set of regular vector fields on $M$ which are generators of a solvable Lie algebra $\mathfrak{h}$. We say that $Y_1,...,Y_r$ are in \emph{canonical  form} if there are $i_1,...,i_l$ such that $i_1+...+i_l=r$ and, for any $x \in M$
$$(Y_1|...|Y_r)=\left(\begin{array}{c|c|c|c}
I_{i_1} & G^1_1(x) & ... & G^1_l(x) \\
\hline
0 & I_{i_2} & ... & G^2_l(x)\\
\hline
\vdots & \ddots & \ddots & \vdots \\
0 & 0 & ... & I_{i_l}\\
\hline
 0 & 0 & 0 & 0 \end{array} \right), $$
where  $G^h_k:M \rightarrow \matr(i_h,i_k)$ are smooth functions.
\end{definition}

\begin{theorem}\label{theorem_solvable_coordinate}
Let $\mathfrak{h}$ be an $r$-dimensional solvable Lie algebra on $M$ such that $\mathfrak{h}$ has constant dimension $r$ as a  distribution of $TM$ then, for any  $x_0 \in M$, there is a set of generators $Y_1,...,Y_r$  of $\mathfrak{h}$ and a local diffeomorphism $\Phi:U(x_0) \rightarrow \tilde{M}$ such that $\Phi_*(Y_1),...,\Phi_*(Y_r)$ are generators in canonical form for $\Phi_*(\mathfrak{h})$.
\end{theorem}
\begin{proof}
The proof can be found in \cite{DMU2}  Theorem 2.6.
${}\hfill$ \end{proof}

\begin{remark} In the particular case of a solvable connected Lie group $\mathcal{H}$ acting  freely and regularly on $M$, Theorem \ref{theorem_solvable_coordinate} admits a (semi)-global version, i.e. we can find a coordinate system in a neighborhood of a set of orbits of the action of the group $\mathcal{H}$. Indeed, if $\mathcal{H}$ acts freely and regularly, $M$ is semi-globally diffeomorphic to $V \times \mathcal{H}$ (where $V \subset \mathbb{R}^{m-r}$) and the generators $Y_1,...,Y_r$ of $\mathfrak{h}$ are vertical vector fields with respect to the bundle structure of $M$. Furthermore, it is possible to choose  a global coordinate system  $g^1,...,g^r$ on $\mathcal{H}$ such that $Y_1,...,Y_r$ are in canonical form (see for example \cite{Onishchik}, Chapter 2 Section 3.1 Corollary 1). In this way we construct a coordinate system in any neighborhood of a set of orbits of $G$.
\end{remark}

\begin{theorem}\label{theorem_infinitesimal_SDE1}
Let $K=\spann \{V_1,...,V_r\}$ be a Lie algebra of $\mathcal{V}(M)$ and suppose that $Y_1,...,Y_r$ is a set of regular vector fields (where $V_i=(Y_i,C_i,\tau_i)$) and that $\mathcal{G}$ is a finite dimensional Lie algebra. Then, for any $x_0 \in M$ there exist an open neighborhood $U$ of $x_0$ and a stochastic transformation $T \in
S(U)$ of the form $T=(Id_U,B,\eta)$ such that $T_*(V_1),...,T_*(V_r)$ are strong infinitesimal stochastic transformations in
$\mathcal{V}(U)$. Furthermore the smooth functions $B,\eta$ are solutions to the equations
\begin{eqnarray}
Y_i(B)&=&-L_{B*}(C_i) \label{equation_reduction1}\\
Y_i(\eta)&=&-\tau_i \eta, \label{equation_reduction2}
\end{eqnarray}
 where $L_g$ is the diffeomorphism given by the left multiplication for $g \in \mathcal{G}$ and  $i=1,...,r$.
\end{theorem}
\begin{proof}
Since $Y_1,...,Y_r$ are regular vector fields for any $x_0 \in M$ there exists a neighborhood $U$ of $x_0$ which is foliated by the orbits of $Y_1,...,Y_r$, i.e. there exists a $M' \subset \mathbb{R}^{m-r}$ and $H \subset \mathbb{R}^{r}$ such that $U=M' \times H$ and such that $\langle Y_1,...,Y_r \rangle \in TH$, i.e. $Y_1,...,Y_r$ are vertical vector field for the trivial bundle $M' \times H$. \\
Furthermore there exists a section of the trivial bundle $M' \times H$, i.e. a smooth map $S:M' \rightarrow U$ such that $S(x')=(x',S^H(x'))$. Let $x'_0 \in M'$ be a fixed point and denote by $H_{x_0'}$ the fiber of $x_0'$ in the trivial bundle $M' \times H$. Consider now the trivial bundle $H_{x_0'} \times (\mathcal{G} \times \mathbb{R}_+)$. We can define a linear map $K:TH_{x_0} \rightarrow \mathfrak{g} \times \mathbb{R}$ i.e. a connection on the trivial bundle $H_{x_0'} \times (\mathcal{G} \times \mathbb{R}_+)$ using the infinitesimal stochastic transformations $V_1,...,V_r$. Indeed if $R \in TH_{x_0'}$, since $Y_1,...,Y_r$ are regular vector fields, $R$ can be written in a unique way as $R=r^{k}(y) Y_k$ where $r^k(y)$ are smooth functions on $H_{x_0'}$. Then we define $K(R)=(r^k(y)C_k,r^k(y)\tau_k)$. We prove that the connection $K$ is flat. In order to prove this it is sufficient to verify that
$$[(Y_i,K(Y_i)),(Y_j,K(Y_j))]=([Y_i,Y_j],K([Y_i,Y_j])),$$
where the commutation on the left hand side of the previous equation is the commutation of the vector fields $(Y_k,K(Y_k)) \in T(H_{x_0'} \times \mathcal{G} \times \mathbb{R}_+)$. If $[Y_i,Y_j]=\lambda_{ij}^kY_k$, by the rule of commutation of infinitesimal stochastic transformations we have
\begin{eqnarray*}
[(Y_i,K(Y_i)),(Y_j,K(Y_j))]&=&([Y_i,Y_j],Y_i(K(Y_j))-Y_j(K(Y_i))-\{K(Y_i),K(Y_j)\})\\
&=&(\lambda^k_{ij}Y_k,(Y_i(C_j)-Y_j(C_i)-\{C_i,C_j\},Y_i(\tau_j)-Y_j(\tau_i)))\\
&=&(\lambda^k_{ij}Y_k,(\lambda^k_{ij}C_k,\lambda^k_{ij}\tau_k))\\
&=&([Y_i,Y_j],K([Y_i,Y_j])).
\end{eqnarray*}
Since $K$ is a flat connection, and since without lost of generality we can suppose that $H$ is simply connected, there exists a unique function $(B(x_0',y),\eta(x_0',y))$ defined on $H_{x_0'}$ and solving the equations \eqref{equation_reduction1} and \eqref{equation_reduction2} on $H_{x_0'}$. Indeed $(B(x_0',y),\eta(x_0',y))$ can be obtained as the parallel transport of the identity from the point $S(x_0')$ into the point $y$ through the flat connection $K$. The flatness of the connection $K$ ensures that $B(x_0',y)$ and $\eta(x_0',y)$ solve equations \eqref{equation_reduction1} and \eqref{equation_reduction2}. By repeating this construction for any $x' \in M'$, since the section $S$ is smooth, we obtain two smooth functions $B(x',y)$ and $\eta(x',y)$ on $U=M'\times H$ solving equations \eqref{equation_reduction1} and \eqref{equation_reduction2}.
${}\hfill$\end{proof}\\

\begin{remark}
It is possible to partially extend the previous theorem to the case in which $\mathcal{G}$ is an infinite dimensional subgroup of the group of diffeomorphisms of $N$. More precisely if $\mathcal{G}_{loc}$ is the subgroupoid of the groupoid of local diffeomorphisms of $N$ such that the Lie algebra of $\mathcal{G}_{loc}$ is generated by the Lie algebra of $\mathcal{G}$, we can prove the existence of $B \in \mathcal{G}_{loc}$ satisfying the thesis of Theorem \ref{theorem_infinitesimal_SDE1}. However sufficient conditions on $V_1,...,V_r$ assuring that $B$ is not only in $\mathcal{G}_{loc}$ but also in $\mathcal{G}$ are not yet worked out.
\end{remark}

\begin{lemma}\label{lemma_reduction}
If $Y_1,...,Y_r$ are a set of vector fields in canonical form which are also strong symmetries of the geometrical SDE $\Psi$, then $\Psi$ is in triangular form with respect to $x^1,...,x^r$.
\end{lemma}
\begin{proof}
Since $Y_1,...,Y_r$ are in canonical form, we have that $Y_i=\sum_{j<i}Y^i_j(x)\partial_{x^j}+\partial_{x^i}$. Using the determining equation for $Y_1$ we obtain $1-\partial_{x^1}(\Psi^1)=0$ and $\partial_{x^1}(\Psi^i)=0$ for $i>1$. This means that $\Psi^1=x^1+B^1(x^2,...,x^m,z)$ and $\Psi^i$ do not depend on $x^1$ for $i>1$. Using an analogous reasoning and a proof by induction we can prove that $\partial_{x^i}(\Psi^i)=1$ and $\partial_{x^i}(\Psi^k)=0$ for $k>j$. This implies the thesis of the lemma.
${}\hfill$\end{proof}\\

\begin{theorem}
Let $V_1=(Y_1,C_1,\eta_1),...,V_r=(Y_r,C_r,\eta_r) \in \mathcal{V}(M)$ be a solvable Lie algebra of symmetries of the SDE $\Psi$ such that $Y_1,...,Y_r$ are regular vector fields and $\mathcal{G}$ is a finite dimensional Lie group. Then, for any $x_0 \in M$, there exist a neighborhood $U$ of $x_0$ and a stochastic transformation $T \in S(U)$ such that $E_T(\Psi)$ is in triangular form.
\end{theorem}
\begin{proof}
We prove that there exists a $T=T_1 \circ T_2$, where $T_1=(\Phi_1,id_{\mathcal{G}}, 1)\in S(U)$ and $T_2=(id_M,B,\eta) \in S(U)$, satisfying the thesis of the theorem. Owing to  Theorem \ref{theorem_infinitesimal_SDE1}, the transformation $T_2$ can be chosen  such that $T_2^*(V_1)=(Y_1,0,0),...,T_2^*(V_r)=(Y_r,0,0)$. This means that $Y_1,...,Y_r$ form a regular solvable Lie algebra of strong symmetries of $E_{T_2}(\Psi)$. By Theorem \ref{theorem_solvable_coordinate} there exists a (locally defined) map $\Phi_1$ such that $\Phi_1^*(Y_1),...,\Phi^*_1(Y_r)$ are in canonical form and they are symmetries of $E_{T_1}(E_{T_2}(\Psi))$. By Lemma \ref{lemma_reduction}, this implies that $E_{T_1}(E_{T_2}(\Psi))$ is in triangular form. Since $E_T(\Psi)=E_{T_1}(E_{T_2}(\Psi))$ the theorem is proved.
${}\hfill$\end{proof}\\

\section{A two dimensional example}\label{subsection_example}

In  order to give an idea of the generality and of the flexibility of our approach, in this section we propose an example of application of the previous theoretical results. For the sake of simplicity we consider only semimartingale with gauge symmetries (no time symmetries) and we use the notation introduced in Remark \ref{remark_senza}.\\

Setting $M=\mathbb{R}^2$, $N=GL(2) \times \mathbb{R}^2$ (with the natural multiplication), we consider  the geometrical  SDE
\begin{equation}\label{equation_example}
\Psi(x,z_{(1)},z_{(2)})=z_{(1)} \cdot x+z_{(2)}.
\end{equation}
The SDE associated with $\Psi$ is an affine SDE and its solution $(X,Z)$ satisfies the following stochastic differential relation
\begin{equation}\label{equation_example1} dX^i_t=X^j_{t_-}(Z^{-1}_{(1)})^k_{j,t_-}{dZ^i_{k,(1),t}+dZ^i_{(2),t}},
\end{equation}
where $Z^{-1}_{(1)}$ is the inverse matrix of $Z_{(1)}$ and  $GL(2)$ is naturally embedded in the set of the two by two matrices. If we set
\begin{equation}\label{equation_example2}
\overline{Z}^i_{j,t}=\int_0^t{(Z^{-1}_{(1)})^k_{j,s_-}{dZ^i_{k,(1),s}}},
\end{equation}
equation \refeqn{equation_example1} becomes the most general equation affine both in the noises $Z_{(1)},Z_{(2)}$ and in the unknown
process $X$. Furthermore, if the noises $Z_{(1)},Z_{(2)}$ are discrete time semimartingales (i.e. semimartingales with fixed time jumps at times
$n \in \mathbb{N}$) equation \refeqn{equation_example1} becomes $X_n=Z_{(1),n-1}^{-1} \cdot Z_{(1),n} \cdot X_{n-1}+Z_{(2),n}-Z_{(2),n-1}$, that
is an affine type iterated
random map (see Subsection \ref{subsection_iterated} and references therein). \\
The SDE $\Psi$ does not have strong symmetries, in the sense that, for general semimartingales $(Z_{(1)},Z_{(2)})$, equation \eqref{equation_example1} does not admit symmetries. \\
For this reason we suppose that the semimartingales $Z_{(1)},Z_{(2)}$ have the gauge symmetry  group $O(2)$ with the natural action
\begin{equation}\label{equation_example3}
\Xi_{B}(z_{(1)},z_{(2)})=(B \cdot z_{(1)} \cdot B^T, B \cdot z_{(2)}),
\end{equation}
where $B \in O(2)$.\\
In order to use the determining equation \refeqn{equation_determining2} for calculating the infinitesimal symmetries of the SDE $\Psi$, we need
to explicitly write the infinitesimal generator $K$ of the action $\Xi_B$ on $N$. In the standard coordinate system of $N$ we have that $\Xi_B$
is generated by
\begin{eqnarray*}
K&=&(-z^2_{1,(1)}-z^1_{2,(1)})\partial_{z^1_{1,(1)}}+(z^1_{1,(1)}-z^2_{2,(1)})\partial_{z^1_{2,(1)}}+(z^1_{1,(1)}-z^2_{2,(1)})\partial_{z^2_{1,(1)}}+\\
&&+(z^1_{2,(1)}+z^2_{1,(1)})\partial_{z^2_{2,(1)}}-z^2_{(2)}\partial_{z^1_{(2)}}+z^1_{(2)}\partial_{z^2_{(2)}}. \end{eqnarray*} If we denote by
$$R=\left(\begin{array}{cc}
0 & -1 \\
1 & 0 \end{array} \right) $$ we have that

\begin{eqnarray*}
K(z_{(1)})&=&R \cdot z_{(1)}+ z_{(1)} \cdot R^T\\
K(z_{(2)})&=&R \cdot z_{(2)},
\end{eqnarray*}
where the vector field $K$ is applied componentwise to the matrix $z_{(1)}$ and the vector $z_{(2)}$. Using this property of $K$ we can easily
prove that
$$V=(Y,C)=\left(-x^2\partial_{x^1}+x^1\partial_{x^2}, 1 \right),$$
(where $C=1$ is the component of the gauge symmetry with respect to the generator $K$) is a symmetry of the equation $\Psi$. Indeed, recalling
that $Y$ is a linear vector field whose components satisfy the relation
$$Y=R \cdot x$$
we have that, in this case, the determining equations \refeqn{equation_determining2} read
\begin{eqnarray*}
Y \circ \Psi-Y(\Psi)-C(x)K(\Psi)&=& R \cdot (z_{(1)} \cdot x+z_{(2)} )-z_{(1)} \cdot (R \cdot x)- K(\Psi)\\
&=&R \cdot (z_{(1)} \cdot x+z_{(2)} )+z_{(1)} \cdot R^T  \cdot x- (R \cdot z_{(1)}+ z_{(1)} \cdot R^T) \cdot x+\\
&&-R \cdot z_{(2)}=0.
\end{eqnarray*}
Since $V$ satisfies the determining equations \refeqn{equation_determining2}, $V$ is an infinitesimal symmetry of $\Psi$. The infinitesimal
stochastic transformation $V$ generates a one-parameter group of symmetries of $\Psi$ given by
$$T_a=(\Phi_a,B_a)=\left(\left(
\begin{array}{cc}
\cos(a) & -\sin(a)\\
\sin(a) & \cos(a) \end{array} \right) \cdot x, \left(
\begin{array}{cc}
\cos(a) & -\sin(a)\\
\sin(a) & \cos(a)
\end{array} \right)  \right).$$
In other words, if the law of $(Z_{(1)},Z_{(2)})$ is gauge invariant with respect to rotations, then the SDE $\Psi$ is invariant with respect to
rotations.
\\
Once we have found an infinitesimal symmetry, we can exploit it to transform the SDE $\Psi$ in an equation of a simpler form as  done, for
example,  in \cite{DMU2} for Brownian-motion-driven SDEs.\\ The first step consists in looking for a stochastic transformation $T=(\Phi,B)$ such
that $T_*(V)$ is a strong symmetry (the existence of the transformation $T$ is guaranteed by Theorem \ref{theorem_infinitesimal_SDE1}). In this
specific case the transformation $T$ has the following form
\begin{equation}\label{equation_transformation}
T=(\Phi(x),B(x))=\left(\left(\begin{array}{c}
x^1 \\
x^2\end{array} \right),\left(\begin{array}{cc}
\frac{x^1}{\sqrt{(x^1)^2+(x^2)^2}} & \frac{x^2}{\sqrt{(x^1)^2+(x^2)^2}}\\
\frac{-x^2}{\sqrt{(x^1)^2+(x^2)^2}} & \frac{x^1}{\sqrt{(x^1)^2+(x^2)^2}}\end{array}\right)\right) \end{equation} and the SDE $\Psi'=E_T(\Psi)$
becomes
$$\Psi'(x,z_{(1)},z_{(2)})=\left(\begin{array}{cc}
x^1 & -x^2\\
x^2 & x^1 \end{array} \right)\cdot \left(\begin{array}{c}
z^{1}_{1,(1)}\\
z^2_{1,(1)} \end{array}\right)+ \left(\begin{array}{cc}
\frac{x^1}{\sqrt{(x^1)^2+(x^2)^2}} & \frac{-x^2}{\sqrt{(x^1)^2+(x^2)^2}}\\
\frac{x^2}{\sqrt{(x^1)^2+(x^2)^2}} & \frac{x^1}{\sqrt{(x^1)^2+(x^2)^2}}\end{array}\right) \cdot \left(\begin{array}{c} z^1_{(2)}\\
z^2_{(2)} \end{array}\right).$$ Note that $\Psi'$ does not depend on $z^{1}_{2,(1)},z^{2}_{2,(1)}$ and so the noise is smaller. So the
transformation $T$ has an effect similar to the reduction of redundant Brownian motions in continuous SDE (see \cite{Elworthy1999}).
 Moreover, if we rewrite the transformed SDE in (pseudo)-polar coordinates
\begin{eqnarray*}
\rho&=&(x^1)^2+(x^2)^2\\
\theta&=&\argg(x^1,x^2),
\end{eqnarray*}
where $\argg(a,b)$ is the function  giving the measure of the angle between $(0,1)$ and $(a,b)$ in $\mathbb{R}^2$, we find
\begin{equation}\label{equation_polar1}
\begin{array}{ccl}
\Psi'^{\rho}(\rho,\theta,z)&=&(\sqrt{\rho}z^1_{1,(1)}+z^1_{(2)})^2+(\sqrt{\rho}z^2_{1,(1)}+z^2_{(2)})^2\\
\Psi'^{\theta}(\rho,\theta,z)&=&\theta+\argg(\sqrt{\rho}z^1_{1,(1)}+z^1_{(2)},\sqrt{\rho}z^2_{1,(1)}+z^2_{(2)}),
\end{array}
\end{equation}
The geometrical SDE defined by $(\Psi^{\rho},\Psi^{\theta})$ is a triangular SDE with respect to the solutions processes $(R_t,\Theta_t)$. Indeed
we have
\begin{equation}\label{equation_R}
\begin{array}{rcl}
 dR_t&=&\left(d\left[Z'^1_{(2)},Z'^1_{(2)}\right]^c_t+d\left[Z'^2_{(2)},Z'^2_{(2)}\right]^c_t+(\Delta Z'^1_{(2),t})^2+(\Delta Z'^1_{(2),t})^2\right)+\\
 &&+\sqrt{R_{t_-}}\left(dZ'^1_{(2),t}+2d\left[Z'^1_{(2)},\overline{Z}'^1_1\right]^c_t+2d\left[Z'^2_{(2)},\overline{Z}'^2_1\right]^c_t+2\Delta
\overline{Z}'^{1}_{1,t} \Delta Z'^1_{(2),t}+2\Delta \overline{Z}'^{2}_{1,t} \Delta Z'^2_{(2),t}\right)+\\
&&+
R_{t_-}\left(d\overline{Z}'^1_{1,t}+d\left[\overline{Z}'^1_1,\overline{Z}'^{1}_1\right]_t^c+d\left[\overline{Z}'^2_1,\overline{Z}'^2_1\right]_t^c+(\Delta
\overline{Z}^{1}_{1,t})^2+(\Delta \overline{Z}^2_{1,t})^2\right)
\end{array}
\end{equation}
\begin{equation}\label{equation_Theta}
\begin{array}{rcl}
d\Theta_t&=&\left(-d\overline{Z}'^2_{1,t}+2d\left[\overline{Z}'^{1}_1,\overline{Z}'^2_1\right]^c_t
\right)+\\
&&+\frac{1}{\sqrt{R}_{t_-}}\left(-dZ'^2_{(2),t}+2d\left[\overline{Z}'^{1}_1,Z'^2_{(2)}\right]^c_t
+2d\left[Z'^{1}_{(2)},\overline{Z}'^2_1\right]^c_t\right)+\frac{2}{R_{t_-}}d\left[Z'^{1}_{(2)},Z'^2_{(2)}\right]^c_t+\\
&&+\left(\argg\left(\sqrt{R_{t_-}}(1+\Delta \overline{Z}'^{1}_{1,t})+\Delta Z'^1_{(2),t},\sqrt{R_{t_-}}(\Delta \overline{Z}'^{2}_{1,t})+\Delta
Z'^2_{(2),t}\right) +\Delta \overline{Z}'^2_{1,t}+\frac{\Delta Z'^2_{(2),t}}{\sqrt{R_{t_-}}}\right),
\end{array}
\end{equation}
where
\begin{eqnarray*}
dZ'^i_{(2),t}&=&B^{i}_j(X_{t_-})dZ^j_{(2),t}\\
dZ'^i_{j,(1),t}&=&Z'^i_{k(1),{t_-}}B^{k}_l(X_{t_-})B^j_r(X_{t_-})(Z^{-1}_{(1)})^{l}_{m,t_-} dZ^m_{r,(1),t}\\
d\overline{Z}'^{i}_{j,t}&=&(Z'^{-1}_{(1)})^i_{k,t_-}dZ'^k_{j(1),t} =B^{i}_k(X_{t_-})B^{j}_r(X_{t_-})d\overline{Z}^{k}_{r,t}.
\end{eqnarray*}
Here $B(x)$ is given in \refeqn{equation_transformation}, $X^1_t=\sqrt{R_t}\cos(\Theta_t)$, $X^2_t=\sqrt{R_t}\sin(\Theta_t)$ and
$\overline{Z}^i_j$ are given in equation \refeqn{equation_example2}. SDEs \refeqn{equation_R} and \refeqn{equation_Theta}
 are in triangular form: indeed, the equation for $R_t$ depends only on $R_t$, while the equation for $\Theta_t$ is independent from $\Theta_t$ itself.
 This means that the process $\Theta_t$ can be reconstructed from the process $R_t$ and the semimartingales $(Z'_{(1)},Z'_{(2)})$ using only
 integrations. Furthermore, using the inverse of the stochastic transformation \refeqn{equation_transformation}, we
 can recover both the solution process $X^1_t,X^2_t$ and the initial noise $(Z_{(1)},Z_{(2)})$  using only inversion of functions and
 It\^o integrations. This situation is very similar to what happens in the
 deterministic setting (see \cite{Olver1993,Stephani1989}) and in the Brownian motion case (see \cite{DMU2}),
 where  the presence of a one-parameter symmetry group allows us  to
split the differential system into a system of lower dimension and an integration (the so called reduction and reconstruction by quadratures).
Also the equation for $R_t$ seems to have a familiar form. Indeed, if $Z_{(1)}=I_2$ (the two dimensional identity matrix) and
$Z_{(2)}$ is a two dimensional Brownian motion, equation \refeqn{equation_R} becomes the equation of the two dimensional Bessel process. This
fact should not surprise us since  the proposed reduction procedure  is the usual reduction procedure of a two dimensional Brownian motion with
respect to the rotation group. For  generic $(Z_{(1)},Z_{(2)})$ the equation for $R_t$ has the form
$$dR_t=d\mathfrak{Z}^1_t+\sqrt{R_{t_-}}d\mathfrak{Z}^2_t+R_{t_-}d\mathfrak{Z}^2_t,$$
 where\small
\begin{eqnarray*}
\mathfrak{Z}^1_t&=&\left[Z'^1_{(2)},Z'^1_{(2)}\right]^c_t+\left[Z'^2_{(2)},Z'^2_{(2)}\right]^c_t+\sum_{0 \leq s \leq t}\left((\Delta
Z'^1_{(2),s})^2+(\Delta Z'^1_{(2),s})^2\right)\\
\mathfrak{Z}^2_t&=& Z'^1_{(2),t}+2\left[Z'^1_{(2)},\overline{Z}'^1_1\right]^c_t+2\left[Z'^2_{(2)},\overline{Z}'^2_1\right]^c_t+\sum_{0 \leq s
\leq t}\left(2\Delta \overline{Z}'^{1}_{1,s} \Delta Z'^1_{(2),s}+2\Delta \overline{Z}'^{2}_{1,s} \Delta Z'^2_{(2),s}\right)\\
\mathfrak{Z}^3_t&=&\overline{Z}'^1_{1,t}+\left[\overline{Z}'^1_1,\overline{Z}'^{1}_1\right]_t^c+
\left[\overline{Z}'^2_1,\overline{Z}'^2_1\right]_t^c+\sum_{0 \leq s \leq t}\left((\Delta \overline{Z}'^{1}_{1,t})^2+(\Delta
\overline{Z}'^2_{1,t})^2\right).
\end{eqnarray*}
\normalsize
Equation \refeqn{equation_R}  can be considered a kind of generalization of affine processes (see \cite{Filipovic2003}): indeed, if
 $Z_{(1)}$ is deterministic and $Z_{(2)}$ is a two dimensional Brownian motion, equation \refeqn{equation_R} reduces to a CIR
model equation with time dependent coefficients.

\section{Weak symmetries of numerical approximations of SDEs driven by Brownian motion \label{section_weak_numerical}}

In this last section we deal with the symmetries of the numerical schemes of Brownian-motion driven SDEs as studied in \cite{DeVecchi2017}, where an invariant numerical integrator for symmetric Brownian motion driven SDEs was introduced by means of the identification of a suitable coordinate system for the discretization. Indeed, in  \cite{DeVecchi2017} the concept of strong symmetry of a SDE and of the related numerical scheme was exploited  in order to provide necessary and sufficient conditions  guaranteeing invariant properties  of the adapted-to-symmetries numerical approximations. Moreover, some theoretical estimates showing the stability and efficiency of the symmetry methods for the class of general linear SDEs were proved. \\
In this section we extend the results of \cite{DeVecchi2017} defining the concept of weak symmetry of a numerical discretization of a Brownian-motion-driven SDE.
Using the theory described in this paper, in the notations of Remark \ref{remark_senza},  a weak symmetry of a numerical discretization of a Brownian-motion-driven SDE is a pair $(\Phi,B)$ where $B \in O(k)$ is a gauge symmetry.\\
On the other hand, there is a natural notion of (weak) symmetry for Brownian-motion-driven SDEs (see \cite{DMU1}) and in the following we discuss the relations between the symmetries of the SDE and the (weak) symmetries of its discretization.  \\
After recalling the notion of (weak) symmetry for Brownian-motion-driven SDEs, we exploit the general theory developed in this paper to describe a numerical scheme for a Brownian-motion-driven SDE
as a geometrical SDE
driven by a semimartingale $Z$ related to the Brownian motion $W$. Once  the gauge symmetries of the semimartingale $Z$ are analyzed, we are able to find the symmetries of the numerical
schemes as the symmetries of the geometrical SDE.\\
In particular we focus on the two more common numerical schemes for equation \refeqn{equation_jump_numerical1}: the Euler and the Milstein discretizations.

\subsection{Symmetries of Brownian-motion-driven SDEs}

In this section we recall the notion of weak symmetry for a Brownian-motion-driven SDE as developed in \cite{DMU1}. This notion can be seen as a special case of the theory described in this paper. The fundamental difference is that, in the Brownian motion case, the determining equations \refeqn{equation_determining2} give only a sufficient condition for finding symmetries. This is due to the fact that, since Brownian motion is a continuous semimartingale, there is a huge set of geometrical SDEs $\Psi$ driven by Brownian motion having the same set of solutions. \\
For this reason hereafter a SDE driven by the Brownian motion is no more identified with a geometrical SDE $\Psi$ but with the more standard pair of coefficients $(\mu,\sigma)$. More precisely the SDE $(\mu,\sigma)$ is an equation of the form

\begin{equation}\label{equation_jump_numerical1}
dX^i_t=\mu^i(X_t) dt+ \sigma_{\alpha}^i(X_t)dW^{\alpha}_t.
\end{equation}
\noindent for $\alpha=1,\dots, k$.
A solution $X$ to the SDE \refeqn{equation_jump_numerical1}  is a diffusion process admitting as infinitesimal generator:
\begin{equation}\label{generator}
L= \frac{1}{2}\sigma^i_{\alpha}\sigma^j_{\alpha}\partial_{x^ix^j}+\mu^i\partial_{x^i}.
\end{equation}

Since Brownian motion has both gauge symmetries (given by the rotations group $O(k)$ with its natural action on $\mathbb{R}^k$) and time symmetries (with action $\Gamma_r(z)=r^{1/2}z$), a weak (infinitesimal) symmetry is given by a triplet $(Y,C,\tau)$ with $C:M\rightarrow \mathfrak{so}(k)$ and $\eta:M\rightarrow \mathbb{R}$. We first recall the following theorem.

\begin{theorem}
Let $V=(Y,C,\tau)$ be an infinitesimal stochastic transformation. Then $V$
is a weak symmetry of the SDE $(\mu,\sigma)$ if and only if
$Y$ generates a one parameter group on $M$ and the following determining equations hold

\begin{equation}\label{equation_BW}
\begin{array}{c}
Y^i\partial_{x^i}(\mu^j)-L(Y^j)+\tau \mu^j = 0 \\
 Y^i\partial_{x^i}(\sigma^j_{\alpha})- \sigma^i_{\alpha}\partial_{x^i}(Y^j)+\frac{1}{2} \tau \sigma^i_{\alpha}+\sigma^j_{\beta}  C^{\beta}_{\alpha} = 0.\\
\end{array}
\end{equation}

\end{theorem}
\begin{proof}
The proof can be found in \cite{DMU1}, Theorem 4.4.
$\hfill$\end{proof}\\

In the following we shall use the name quasi strong (infinitesimal) transformations or symmetries for stochastic transformations $V$ of the form $V=(Y,C,0)$.

\subsection{Symmetries of the Euler scheme}

The Euler scheme for equation \refeqn{equation_jump_numerical1}, given a partition $\{t_{\ell}\}_{\ell}$ of $[0,T]$, reads
\begin{equation}\label{eulerscheme}
X^{i,N}_{t_{\ell}}=X^{i,N}_{t_{\ell-1}}+\mu(X^N_{t_{\ell-1}})\Delta t_{\ell}+\sum_{\alpha=1}^k\sigma^i_{\alpha}(X^{N}_{t_{\ell-1}})\Delta W^{\alpha}_{\ell}
\end{equation}
where $i=1, \ldots, n$, $\Delta t_{\ell}=t_{\ell}-t_{\ell-1}$ and $\Delta W^{\alpha}_{\ell}=W^{\alpha}_{t_{\ell}}-W^{\alpha}_{t_{\ell-1}}$. We define a semimartingale $Z$ taking values in $N=\mathbb{R}^{k+1}$ by putting
\begin{eqnarray*}\label{semimartingale}
Z^0_t&=&\sum_{\ell=0}^{\infty}t_{\ell} I_{[t_{\ell} ,  t_{\ell+1})}(t)\\
Z^{\alpha}_t&=&\sum_{\ell=0}^{\infty}W^{\alpha}_{t_{\ell}}I_{[t_{\ell} ,  t_{\ell+1})}(t).
\end{eqnarray*}
Since $Z_t$ is a process with predictable jumps at $t_{\ell}$, $Z_t$ is a semimartingale. Let us introduce the maps
$$F^i(x,z)=x^i+\mu(x)z^0+\sigma^i_{\alpha}(x)z^{\alpha}.$$
It is clear that $X^N_t$, solution to \eqref{eulerscheme}, can be seen as the solution to the geometrical  SDE defined by $F$ and driven by the above defined semimartingale $Z$ (according to \eqref{equation_manifold2} with $\Psi^i=F^i$).\\

Let $\Xi_B$ be an action of $O(k)$ on $\mathbb{R}^{k+1}$ such that $\Xi_{B}(z)=(z^0,B^{\alpha}_{\beta}z^{\beta})$, where $  ,B^{\alpha}_{\beta} $ is the $\alpha,\beta$ element of a square matrix of dimension $k$.
The following theorem holds.

\begin{theorem}\label{theorem_discretization1}
The group $O(k)$ with action $\Xi_B$ is a gauge symmetry group for $Z$ with respect to its natural filtration.
\end{theorem}
\begin{proof}
We directly use the definition of gauge symmetry group. Let $\mathcal{F}^W_t$ be the natural filtration of the $k$ dimensional Brownian motion $W$. Let $B_t$ be a predictable locally bounded process taking values in $O(k)$ with respect to the natural filtration
of $Z$  $\mathcal{F}^Z_t$ . This means that $B_{t_{\ell}}$ is $\mathcal{F}^W_{t_{\ell-1}}$ measurable. Moreover, the process $\tilde{B}_t$, defined by
$$\tilde{B}_t=\sum_{\ell=0}^{\infty}B_{t_{\ell}}I_{(t_{\ell-1},t_{\ell}]}(t)$$
is an $\mathcal{F}^W_t$-predictable locally bounded process taking values in $O(k)$. Since the rotations are a gauge symmetry group for the $k$ dimensional Brownian motion with respect to its natural filtration, we have that
\begin{equation}\label{equationone}
\tilde{W}^{\alpha}_{t}= \int_0^{t}{\tilde{B}^{\alpha}_{\beta,s}dW^{\beta}_s} =\sum_{\ell=1}^{+\infty}
B^{\alpha}_{\beta,t_{\ell}}(W^{\beta}_{t_{\ell} \wedge
t}-W^{\beta}_{t_{\ell-1} \wedge t})
\end{equation}
is a $k$ dimensional Brownian motion.
In particular the following discretization
\begin{equation}\label{equationtwo}
\tilde{Z}^{\alpha}_t=\sum_{\ell=0}^{\infty}\tilde{W}^{\alpha}_{t_{\ell}}I_{[t_{\ell} ,  t_{\ell+1})}(t)
\end{equation}
has the same law as $Z$ (as can be immediately seen by \eqref{semimartingale} and \eqref{equationtwo}). Since
\begin{eqnarray*}
\tilde{Z}^{\alpha}_{t_{i}}&=&\int_0^{t_i}{B^{\alpha}_{\beta,s}dZ^{\beta}_s}=\sum_{\ell=1}^{i}B^{\alpha}_{\beta,t_{\ell}} (Z^{\beta}_{t_{\ell-1}}-Z^{\beta}_{t_{\ell}})\\
&=&\sum_{\ell=1}^{i}B^{\alpha}_{\beta,t_{\ell}} (W^{\beta}_{t_{\ell}}-W^{\beta}_{t_{\ell-1}})=\tilde{W}^{\alpha}_{t_i}
\end{eqnarray*}
and $B_t$ is a generic predictable locally bounded process with respect to $\mathcal{F}^Z_t$, the thesis follows.
${}\hfill$\end{proof}\\

\noindent Using the previous discussion and Theorem \ref{theorem_discretization1} we can introduce the  concept of weak symmetry of the Euler
discretization scheme $F(x,z)$. Indeed the weak stochastic transformation $T=(\Phi(x),B(x))$, which acts on the solution to the Euler
discretization scheme in the following way $(X'^N,Z')=P_T(X,Z)$ and precisely as
\begin{eqnarray*}
X'^N_{t_{\ell}}&=&\Phi(X^N_{t_{\ell}})\\
\Delta Z'^{\alpha}_{t_{\ell}} &=&B^{\alpha}_{\beta}(X^N_{t_{\ell-1}}) \Delta Z^{\beta}_{t_{\ell}}=
B^{\alpha}_{\beta}(X^N_{t_{\ell-1}}) \Delta W^{\beta}_{\ell}\\
\Delta Z^0_{t_{\ell}}&=&\Delta t_{\ell}.
\end{eqnarray*}
is a symmetry of the Euler discretization scheme if $(X',Z')$ is also a solution to the discretization scheme defined by $F$. By Theorem
\ref{theorem_symmetry1}, a sufficient
condition to be a symmetry of $F$ is
\begin{equation}\label{equation_euler_scheme1}
\Phi(F(\Phi^{-1}(x),\Delta t, (B \circ \Phi^{-1}(x))^{-1} \cdot \Delta W)=F(x,\Delta t,\Delta W),
\end{equation}
where $z=(\Delta t,\Delta W)$. For a given weak infinitesimal stochastic transformation $(Y,C)$ the determining equations reads
\begin{equation}
Y^j(x)\partial_{x^j}(F^i)(x,\Delta t, \Delta W)-F^j(x,\Delta t, \Delta W) \partial_{x^j}(Y^i)(x)
=-C^{\alpha}_{\beta}(x)\Delta W^{\beta} \partial_{\Delta W^{\alpha}}(F^i)(x,\Delta t ,\Delta W).
\end{equation}

The following theorem proposes a generalization of Theorem 3.1 in \cite{DeVecchi2017} to the case of weak stochastic transformations.

\begin{theorem}\label{theorem_euler2}
Let $V=(Y,C,0)$ be a quasi strong symmetry of the SDE  \refeqn{equation_jump_numerical1}. When $Y^i_j=Y_j(x^i)$ are polynomials of first degree in $x^1,...,x^m$,
then $\tilde{V}=(Y,C) \in \mathcal{V}(M)$ is a weak symmetry of the Euler discretization scheme $F$. If, for a given $x_0 \in M$,
$\spann\{\sigma_1(x_0),\dots,\sigma_{m}(x_0) \}=\mathbb{R}^n$, also the converse holds.
\end{theorem}
\begin{proof}
The proof is similar to the one of Theorem 3.1 in \cite{DeVecchi2017}. 
The determining equations for the special case of the Euler discretization method are
\begin{equation}\label{equation_weak_euler4}
Y^i\left(x+\mu \Delta t+ \sigma_{\alpha} \Delta W^{\alpha} \right)-Y^i- Y^j \partial_{x^j}(\mu^i)\Delta t - Y^j \partial_{x^j}(\sigma^i_{\alpha})\Delta W^{\alpha}- \sigma^i_{\beta} C^{\beta}_{\alpha}\Delta W^{\alpha}=0.
\end{equation}
Using equations \refeqn{equation_BW} in equation \refeqn{equation_weak_euler4} we obtain
\begin{equation}\label{equation_weak_euler5}
Y^i\left(x+\mu \Delta t+ \sigma_{\alpha} \Delta W^{\alpha} \right) = Y^i + \mu^j \partial_{x^j}(Y^i) \Delta t+ A^{lk}\partial_{x^kx^l}(Y^i)
\Delta t+ \sigma^j_{\alpha}\partial_{x^j}(Y^i) \Delta W^{\alpha}.
\end{equation}
If $Y^i$ is linear in $x^j$, equation \refeqn{equation_weak_euler5} is satisfied and $(Y,C)$ is a weak symmetry of the Euler discretization scheme. Conversely, if $(Y,C)$ is a symmetry of the Euler discretization scheme and the condition assumed for $\sigma_{\alpha}$ holds, equation \refeqn{equation_weak_euler5} implies that $Y^i$ is linear in $x^j$.
${}\hfill$\end{proof}

\subsection{Symmetries of the Milstein scheme}

The Milstein scheme has the form
\begin{equation}\label{equation_Milstein}
\begin{array}{rcl}
\bar{X}^{i,N}_{t_{\ell}}&=&\bar{X}^{i,N}_{t_{\ell-1}}+\mu^i(\bar{X}_{t_{\ell-1}}) \Delta t_{\ell} +\sum_{\alpha=1}^k \sigma^i_{\alpha}(\bar{X}_{t_{\ell-1}}) \Delta W^{\alpha}_{\ell}+\\
&&+\frac{1}{2}\sum_{j=1}^m\sum_{\alpha,\beta=1}^k\sigma^j_{\alpha}(\bar{X}_{t_{\ell-1}}) \partial_{j}(\sigma^i_{\beta})(\bar{X}_{t_{\ell-1}}) \Delta
\mathbb{W}^{\alpha,\beta}_{\ell},
\end{array}
\end{equation}
where $\Delta \mathbb{W}^{\alpha,\beta}_{\ell}=\int_{t_{\ell-1}}^{t_{\ell}}{(W^{\beta}_s-W^{\beta}_{t_{\ell-1}})dW^{\alpha}_s}$. We recall that when $n=1$ we
have that
$$\Delta \mathbb{W}^{1,1}_{\ell}=\frac{1}{2}((\Delta W_{\ell})^2-\Delta t_{\ell}).$$

We cannot consider the scheme as a geometrical SDE driven by a semimartingale in $\mathbb{R}^{k+1}$ anymore because in this case the driving noise is composed by both the discretization of the Brownian motion $W$ and by the iterated integral $\mathbb{W}^{\alpha,\beta}_t=\int_0^t{W^{\alpha}_sdW^{\beta}_s}$. Therefore in this case
$$N=\mathbb{R} \oplus \mathbb{R}^k \oplus (\mathbb{R}^k \otimes \mathbb{R}^k),$$
and the semimartingale $(t,W_t,\mathbb{W}_t)$ lives exactly in $N$. The vector space $N$ has a natural Lie group structure with composition given by
$$(\alpha_1,a_1,b_1) \circ (\alpha_2,a_2,b_2)=(\alpha_1+\alpha_2,a_1+a_2,a_2 \otimes a_1 + b_1+ b_2),$$
where $\alpha_1,\alpha_2 \in \mathbb{R}$, $a_1,a_2 \in \mathbb{R}^k$ and $b_1,b_2 \in \mathbb{R}^k \otimes \mathbb{R}^k$. In this case $1_N=(0,0,0 \otimes 0)$, while the inverse operation is given by
$$(\alpha_1,a_1, b_1)^{-1}=(-\alpha_1,-a_1,-b_1+a_1 \otimes a_1).$$
Let $Z=(Z^0,Z_1^{\alpha},Z_2^{\alpha,\beta})$ be the semimartingale given by the discretization of $(t,W_t,\mathbb{W}_t)$, that is
\begin{eqnarray*}
Z^0_t&=&t_{\ell} \text{ if }t_{\ell} \leq t < t_{\ell+1}\\
Z^{\alpha}_{1,t}&=&W^{\alpha}_{t_{\ell}} \text{ if }t_{\ell} \leq t < t_{\ell+1}\\
Z^{\alpha,\beta}_{2,t}&=&\int_0^{t_{\ell}}{W^{\alpha}_{s}dW^{\beta}_s} \text{ if }t_{\ell} \leq t < t_{\ell+1}.
\end{eqnarray*}
It is simple to see that
\begin{eqnarray*}
Z_{t_{\ell}} \circ Z^{-1}_{t_{\ell-1}}&=&\left(t_{\ell}-t_{\ell-1},W^{\alpha}_{t_{\ell}}-W^{\alpha}_{t_{\ell-1}},\int_0^{t_{\ell}}{W^{\alpha}_sdW^{\beta}_s}+\right.\\
&&\left.-\int_0^{t_{\ell-1}}{W^{\alpha}_sdW^{\beta}_s}-W^{\alpha}_{t_{\ell-1}}W^{\beta}_{t_{\ell}}+W^{\alpha}_{t_{\ell-1}}W^{\beta}_{t_{\ell-1}}\right)\\
&=&\left(\Delta t_{\ell},\Delta W^{\alpha}_{\ell},\int_{t_{\ell-1}}^{t_{\ell}}{(W^{\alpha}_s-W^{\alpha}_{t_{\ell-1}})dW^{\beta}_s}\right)\\
&=&(\Delta t_{\ell},\Delta W^{\alpha}_{\ell}, \Delta \mathbb{W}^{\alpha\beta}_{\ell}).
\end{eqnarray*}
It is possible to define, in a natural way, an action $\Xi_B:N \rightarrow N$ of the orthogonal matrices $B \in O(k)$ on $N$. Indeed if $B \in N$ we define $\Xi_B$ as the unique linear map from $N$ into itself such that
$$\Xi_B((\alpha,a,b \otimes c))=(\alpha,B \cdot a, (B \cdot b) \otimes (B \cdot c)),$$
for any $a,b,c \in \mathbb{R}^k$.

\begin{theorem}\label{theorem_discretrization2}
The Lie group $O(k)$ with action $\Xi_B$ is a gauge symmetry group for the discretization $Z=(Z^0,Z_1^{\alpha},Z_2^{\alpha,\beta})$ of
$(t,W_t,\mathbb{W}_t)$, with respect to its natural filtration.
\end{theorem}
\begin{proof}
Let us define
$$W'_t=\sum_{t_k \leq t} B_{t_k} \cdot \Delta W_{k}+ B_{t_{\ell}}  (W_t-W_{t_{\ell}}),$$
where $t_k$ is the last discrete time lower than $t$. The thesis of the theorem is equivalent to prove that
$$\mathbb{W}'^{\alpha\beta}_{t_{\ell}}=\int_0^{t_{\ell}}{W'^{\alpha}_sdW'^{\beta}_s }=Z'^{\alpha\beta}_{2,t_{\ell}}$$
where $dZ'_t=\Xi_{B}(dZ_t)$, that is, in particular, the third component of $Z'_{t_{\ell}}$  is given by
\begin{eqnarray*}
Z'^{\alpha\beta}_{2,t_{\ell}}&=&\sum_{k \leq \ell} B^{\alpha}_{\gamma,t_k}B^{\beta}_{\delta,t_k} \Delta Z_{2,t_{k}}^{\gamma\delta}+\sum_{h \leq k < \ell}
B^{\alpha}_{\gamma,t_h}B^{\beta}_{\delta,t_k}\Delta Z_{1,t_h}^{\gamma} \Delta Z^{\delta}_{1,t_k}\\
&=&\sum_{k \leq \ell} B^{\alpha}_{\gamma,t_k}B^{\beta}_{\delta,t_k} \Delta \mathbb{W}^{\gamma\delta}_{k}+\sum_{h \leq k < \ell}
B^{\alpha}_{\gamma,t_h}B^{\beta}_{\delta,t_k}\Delta W_{h}^{\gamma} \Delta W^{\delta}_{k}.
\end{eqnarray*}
We have that
\begin{eqnarray*}
\mathbb{W}'^{\alpha\beta}_{t_{\ell}}&=&\int_0^{t_{\ell}}{W'^{\alpha}_sdW'^{\beta}_s }\\
&=&\int_{0}^{t_{\ell}}{\left(\sum_{t_k \leq s} B^{\alpha}_{\gamma,s} \cdot \Delta W^{\gamma}_{k}+ B^{\alpha}_{\gamma,t_{\ell}}  (W^{\gamma}_s-W^{\gamma}_{t_k}) \right)dW'^{\beta}_s}\\
&=&\sum_{h \leq k < \ell}B^{\alpha}_{\gamma,t_h}\Delta W^{\gamma}_h \Delta W'^{\beta}_k + \sum_{k \leq
\ell}B^{\alpha}_{\gamma,t_k}\int_{t_{k-1}}^{t_k}{(W^{\gamma}_s-W^{\gamma}_{t_{k-1}})dW'^{\beta}_s}\\
&=&\sum_{h \leq k < \ell}B^{\alpha}_{\gamma,t_h}B^{\beta}_{\delta,t_k}\Delta W^{\gamma}_h \Delta W^{\delta}_k + \sum_{k \leq
\ell}B^{\alpha}_{\gamma,t_k}B^{\beta}_{\delta,t_k}\int_{t_{k-1}}^{t_k}{(W^{\gamma}_s-W^{\gamma}_{t_{k-1}})dW^{\delta}_s}\\
&=&\sum_{h \leq k < \ell}B^{\alpha}_{\gamma,t_h}B^{\beta}_{\delta,t_k}\Delta W^{\gamma}_h \Delta W^{\delta}_k + \sum_{k \leq
\ell}B^{\alpha}_{\gamma,t_k}B^{\beta}_{\delta,t_k}\Delta \mathbb{W}^{\gamma\delta}_k=Z'^{\alpha,\beta}_{2,t_{\ell}}.
\end{eqnarray*}
${}\hfill$\end{proof}\\

\noindent Thanks to Theorem \ref{theorem_discretrization2} we can introduce the concept of weak symmetry of a Milstein type discretization scheme.
 We can see the solution $\bar{X}_{\ell}$ of the Milstein scheme \refeqn{equation_Milstein} as an iterated random map defined by the geometrical
 SDE $F(x,z)=F(x,\Delta t,\Delta W,\Delta
\mathbb{W})$, where $F$ has the form
\begin{eqnarray*}
F^i(x,\Delta t, \Delta W, \Delta \mathbb{W})&=&x^i+\mu^i(x) \Delta t_n +\sum_{\alpha=1}^k \sigma^i_{\alpha}(x) \Delta W^{\alpha}+\\
&&+\frac{1}{2}\sum_{j=1}^n\sum_{\alpha,\beta=1}^k\sigma^j_{\alpha}(x) \partial_{j}(\sigma^i_{\beta})(x) \Delta \mathbb{W}^{\alpha,\beta},
\end{eqnarray*}
and driven by the semimartingale $Z$ on the Lie group $N$ with its natural composition described
above.\\
The weak stochastic transformation $T=(\Phi(x),B(x))$ acts on the solution to the Milstein discretization scheme in the following way
$(X'^N,Z')=P_T(X^N,Z)$, that is
\begin{eqnarray*}
\bar{X}'^N_{t_{\ell}}&=&\Phi(\bar{X}^N_{t_{\ell}})\\
\Delta Z'^{\alpha}_{1,t_{\ell}} &=&B^{\alpha}_{\beta}(\bar{X}^N_{t_{\ell-1}})
\Delta Z^{\beta}_{1,t_{\ell}}=B^{\alpha}_{\beta}(\bar{X}^N_{t_{\ell-1}}) \Delta W^{\beta}_{\ell}\\
\Delta Z'^{\alpha\beta}_{2,t_{\ell}}
 &=& B^{\alpha}_{\gamma}(\bar{X}^N_{t_{\ell-1}}) B^{\beta}_{\delta}(\bar{X}^N_{t_{\ell-1}}) \Delta Z^{\gamma\delta}_{2,t_{\ell}} =B^{\alpha}_{\gamma}(\bar{X}^N_{t_{\ell-1}})
B^{\beta}_{\delta}(\bar{X}^N_{t_{\ell-1}}) \Delta \mathbb{W}^{\gamma\delta}_{\ell}\\
\Delta Z^0_{t_{\ell}}&=&\Delta t_{\ell},
\end{eqnarray*}
is a symmetry of the discretization scheme if $(X',Z')$ is also a solution to the discretization scheme $F$. Using Theorem \ref{theorem_symmetry1}, a sufficient
condition for having such a symmetry is that
\begin{equation}\label{equation_milstein_scheme1}
\Phi(F(\Phi^{-1}(x),\Delta t, (B \circ \Phi^{-1})^{-1}(x) \cdot \Delta W, (B \circ \Phi^{-1})^{-1}(x) \otimes (B \circ \Phi^{-1})^{-1}(x) \cdot
\Delta \mathbb{W})=F(x,\Delta t,\Delta W, \Delta \mathbb{W}).
\end{equation}

For a given infinitesimal stochastic transformation $(Y,C)$ the determining equations (using equation \refeqn{equation_determining2}) read
\begin{equation}\label{equation_milstein1}
\begin{array}{c}
Y^j(x)\partial_{x^j}(F^i)(x,\Delta t, \Delta W, \Delta \mathbb{W})-F^j(x,\Delta t, \Delta W, \Delta \mathbb{W}) \partial_{x^j}(Y^i)(x))=\\
=-C^{\alpha}_{\beta}(x)\Delta W^{\beta}
\partial_{\Delta W^{\alpha}}(F^i)(x,\Delta t ,\Delta W, \Delta \mathbb{W})
-C^{\alpha}_{\beta}(x) \Delta\mathbb{W}^{\beta\gamma}
\partial_{\mathbb{W}^{\alpha\gamma}}(F^i)(x,\Delta t ,\Delta W, \Delta \mathbb{W})+\\
-C^{\alpha}_{\beta}(x) \Delta \mathbb{W}^{\beta\gamma}
\partial_{\mathbb{W}^{\gamma\alpha}}(F^i)(x,\Delta t ,\Delta W, \Delta \mathbb{W}).
\end{array}
\end{equation}

We have the following application  of Theorem \ref{theorem_euler2} for the Milstein case.

\begin{theorem}\label{theorem_milstein2}
Suppose that $(Y,C,0)$ is a quasi strong symmetry of the SDE  $(\mu,\sigma)$,  $Y^i$ is linear in $x^i$ and  $C$ is a constant matrix. Then $(Y,C)$ is a
weak symmetry of the Milstein discretization scheme for $(\mu,\sigma)$.
\end{theorem}
\begin{proof}
From Theorem \ref{theorem_euler2}, equation \refeqn{equation_milstein1} and under the hypothesis of the present theorem, we have that $(Y,C)$ is a symmetry of
the Milstein discretization if and only if
\begin{equation}\label{equation_milstein2}
\partial_{x^j}(Y^k)\sigma^i_{\alpha}\partial_{x^i}(\sigma^j_{\beta})\Delta \mathbb{W}^{\alpha\beta}=Y^j
\partial_{x^j}(\sigma^i_{\alpha}\partial_{x^i}(\sigma^k_{\beta}))\Delta \mathbb{W}^{\alpha\beta}+
\sigma^i_{\alpha}\partial_{x^i}(\sigma^j_{\beta})(C^{\alpha}_{\gamma}\Delta \mathbb{W}^{\gamma\beta}+C^{\beta}_{\delta}\Delta \mathbb{W}^{\alpha
\delta}).
\end{equation}
Furthermore $(Y,C,0)$ is a quasi-strong symmetry of the SDE $(\mu,\sigma)$ if and only if
\begin{equation}\label{equation_milstein3}
-Y^i\partial_{x^i}(\sigma^j_{\alpha})+\sigma^i_{\alpha}\partial_{x^i}(Y^j)=C^{\beta}_{\alpha}\sigma^j_{\beta}
\end{equation}
If we derive equation \refeqn{equation_milstein3} with respect to $\sigma^i_{\gamma}\partial_{x^i}$ we obtain
\begin{equation}\label{equation_milstein4}
\begin{array}{c}
-\sigma^i_{\alpha}\partial_{x^i}(Y^j)\partial_{x^j}(\sigma^k_{\beta})-Y^j\sigma^i_{\alpha}\partial_{x^ix^j}(\sigma^k_{\beta})
+\sigma^i_{\alpha}\partial_{x^i}(\sigma^j_{\beta})\partial_{x^j}(Y^k)+\sigma^i_{\alpha}\sigma^j_{\beta}\partial_{x^ix^j}(Y^k)=\\
=\sigma^i_{\alpha}\partial_{x^i}(C^{\gamma}_{\beta})\sigma^k_{\gamma}+\sigma^i_{\alpha}\partial_{x^i}(\sigma^k_{\gamma})C^{\gamma}_{\beta}.
\end{array}
\end{equation}
Using the fact that $Y^i$ is almost linear in $x^i$, and so $\partial_{x^ix^j}(Y^k)=0$, the fact that $C^{\alpha}_{\beta}$ is constant, and so
$\partial_{x^i}(C^{\gamma}_{\beta})=0$, and replacing equation \refeqn{equation_milstein3} in equation \refeqn{equation_milstein4} we obtain the
relation \refeqn{equation_milstein2} which means that $(Y,C)$ is a weak symmetry of the Milstein discretization scheme.
${}\hfill$\end{proof}

\begin{remark}
There is an important difference between Theorem \ref{theorem_euler2} and Theorem \ref{theorem_milstein2}. Indeed Theorem
\ref{theorem_euler2} gives a necessary and sufficient condition in order that a quasi-strong symmetry $(Y,C,0)$ of the SDE $(\mu,\sigma)$ is a weak
symmetry $(Y,C)$ of the Euler discretization scheme, while Theorem \ref{theorem_milstein2} gives only a sufficient condition. Furthermore the
hypotheses of Theorem \ref{theorem_milstein2} request that $C^{\alpha}_{\beta}$ is a constant, while $C^{\alpha}_{\beta}$ can be any function in
Theorem \ref{theorem_euler2}. \\
This last fact can be explained in the following way: the gauge transformation $\Xi_B$ transforms the process
$Z=(Z^{0},Z^{\alpha}_1,Z^{\alpha\beta}_2)$ using an Euler approximation of the usual random rotation and not a Milstein one. The two
approximations of the random rotation coincide only when the rotations $B^{\alpha}_{\beta}$ (or the generator $C^{\alpha}_{\beta}$ of the
rotations) are constants. Finally it is important to note that one cannot use the Milstein approximation for transforming the semimartingales $Z$,
because otherwise the transformation would not preserve the law of the semimartingale $Z$.
\end{remark}

\section*{Acknowledgements}

The first author would like to thank the Department of Mathematics, Universit\`a degli Studi di
Milano for the warm hospitality. This research was supported by IAM and HCM (University of
Bonn) and Gruppo Nazionale Fisica Matematica (GNFM-INdAM). The second author is supported
by the German Research Foundation (DFG) via CRC 1060.

\bibliographystyle{plain}

\bibliography{SymmetriesBib1.bib}

\end{document}